\documentclass[3p]{elsarticle}

\usepackage[utf8]{inputenc}
\usepackage[T1]{fontenc}
\usepackage[english]{babel}                 
\usepackage{amsmath}
\usepackage{graphicx}
\usepackage{caption}
\usepackage{subcaption}
\usepackage{listings}
\usepackage{algorithm,algorithmic}
\usepackage{bm}
\usepackage{color}

\usepackage{tikz}
\usetikzlibrary{positioning,shapes,arrows,calc,shapes.geometric,backgrounds,fit}

\newtheorem{theorem}{Theorem}

\usepackage{amssymb}

\newenvironment{proof}[1][Proof]{\begin{trivlist}
\item[\hskip \labelsep {\bfseries #1}]}{\end{trivlist}}

\newenvironment{example}[1][Example]{\begin{trivlist}
\item[\hskip \labelsep {\bfseries #1}]}{\end{trivlist}}

\DeclareMathOperator*{\argmin}{arg\,min}

\let\div\undefined
\DeclareMathOperator{\div}{div}
\DeclareMathOperator{\tandiv}{div_{\Gamma}}

\journal{arXiv}

\begin{document}

\begin{frontmatter}

\title{An Approximate Newton Smoothing Method for Shape Optimization}

\author[adressJonas]{Jonas Kusch\corref{mycorrespondingauthor}}
\cortext[mycorrespondingauthor]{Corresponding author}
\author[adressStephan]{Stephan Schmidt}
\author[adressNico]{Nicolas R. Gauger}

\address[adressJonas]{Karlsruhe Institute of Technology, Kaiserstr. 12, 76131 Karlsruhe, Germany}
\address[adressStephan]{Universit\"at W\"urzburg, Emil-Fischer-Str. 40, 97074 W\"urzburg, Germany}
\address[adressNico]{TU Kaiserslautern, Paul-Ehrlich-Str. 34, 67663 Kaiserslautern, Germany}

\begin{abstract}
A novel methodology to efficiently approximate the Hessian for numerical shape optimization is considered. The method enhances operator symbol approximations by including body fitted coordinates and spatially changing symbols in a semi automated framework based on local Fourier analysis. Contrary to classical operator symbol methods, the proposed strategy will identify areas in which a non-smooth design is physically meaningful and will automatically turn off smoothing in these regions. A new strategy to also numerically identify the analytic symbol is derived, extending the procedure to a wide variety of problems. The effectiveness is demonstrated by using drag minimization in Stokes and Navier--Stokes flows.

\end{abstract}


\end{frontmatter}


\section{Introduction}
\label{sec:section1}

Shape optimization subject to partial differential equations plays an important role in a variety of problems such as minimum drag shapes in fluid dynamics, acoustics, material sciences or geometric inverse problems in non-destructive testing and medical imaging. 

In order to propose an efficient design, an initial geometry is described with the help of a finite set of parameters, which are modified such that a given cost function is minimized. Steepest-descend methods iteratively modify the design according to the negative gradient with respect to the chosen parameters, thereby ensuring an successive descend of the cost function. The use of the adjoint approach~\cite{jameson1988aerodynamic} makes computing the gradient independent of the number of design parameters, promoting using all mesh node positions as design parameters, i.e., the richest design space possible. A downside of the plethora of design variables are possibly high-frequency oscillations in the search direction as no boundary smoothness is inherent in the parametrization. Consequently, resulting non-smooth designs can cause an irregular computational mesh resulting in the failure of the optimization. 

A natural choice to overcome these difficulties is using a smoothed~\cite{jameson1988aerodynamic,jameson1990automatic,jameson1994optimum} or, equivalently, a Sobolev gradient descend~\cite{renka2006simple}, which necessitates a manual or automatic parameter study to determine a problem-dependent smoothing parameter \cite{kim2005enhancement}. Picking an adequate smoothing parameter is a crucial task as manipulating the search direction poorly can potentially slow down the convergence. 

Additionally, the convergence speed of the steepest-Descent method deteriorates if the Hessian of the optimization problem is ill-conditioned~\cite[Chapter~3.3]{wright1999numerical}. An approach to overcome the condition number dependency is Newton's method, necessitating a computation of the Hessian. As it is computationally not feasible to determine the exact Hessian, a lot of work has gone into its approximation. One strategy is to approximate the symbol of the exact operator \cite{arian1995analysis,arian1999analysis,arian1999preconditioning}. Furthermore, closely related studies based on Fourier analysis have also been used in~\cite{yang2011shape} to study the condition of several Navier--Stokes flow situations, bridging the gap between optimization acceleration and studying the well- or ill-posed nature of flow problems.

An approach to construct a search direction fulfilling both, the desired regularity of the design as well as including Hessian information has been derived in  \cite{schmidt2009impulse} for energy minimization. The authors use the symbol of the exact Hessian for a half-space geometry to choose a constant parameter for Sobolev smoothing which is based on the spacing of the computational mesh. The choice of a constant smoothing parameter can however lead to a limitation of the design space, as non-smooth areas can become physically meaningful in certain areas of the design. One can for example think of the sharp trailing edge of an airfoil. Furthermore, the limitation to half-space geometries means that the smoothing parameter might not be valid in practical applications. 

The aim of this paper is to extend the derivation of the Hessian symbol to body fitted coordinates, allowing the direct usage of the symbol in the construction of an approximate Newton method. The resulting preconditioner will inherit the local smoothing properties of the exact Hessian by picking spatially dependent smoothing parameters automatically. 
The paper is structured as follows: After the introduction in section \ref{sec:section1}, we derive the steepest-descent search direction in section \ref{sec:section2}. To accelerate the optimization, we derive the Hessian symbol in section \ref{sec:section3} and discuss why the Hessian has smoothing behavior. This behavior is demonstrated and investigated further in section \ref{sec:section4}, where a numerical approximation of the Hessian symbol for low Reynolds number flows is presented. The considered technique to approximate this symbol gives insight into the smoothing properties of the Hessian. Having derived the analytic symbol, we construct a preconditioner, which approximates this symbol in section \ref{sec:section5}. To ensure minimal computational costs, we use differential operators to construct said preconditioner. The coefficients of these operators are determined automatically to match the exact Hessian symbol. Finally, we compare our novel method in section \ref{sec:section6} to classical Sobolev smoothing when using a constant smoothing parameter.

\section{The steepest-descent search direction}
\label{sec:section2}
We start by defining the shape optimization problem for a general cost function $F$ given a Stokes flow: This constraint optimization problem takes the form
\begin{subequations}\label{eq:optProblem}
\begin{align}
\min_{\bm{v},p,\Gamma_{o}}F(\bm{v},p,&\Gamma_{o}) \\
\text{s.t. }
-\mu \Delta \bm{v} + \nabla p &= 0,\\
\nabla \cdot \bm{v} &= 0,\\
\bm{v} = 0 &\text{  on } \Gamma_{o},
\end{align}
\end{subequations}
where the design variable $\Gamma_{o}$ is the surface of a flow obstacle with volume $\Omega_o$. In our setting, the flow domain is given by $\Omega = \mathbb{R}^d\setminus \Omega_o$. The velocity $\bm{v}\in\mathbb{R}^2$ and the pressure $p\in\mathbb{R}$ fulfill the Stokes equations with dynamic viscosity $\mu$. Physically, the Stokes equations describe a creeping flow, in which convective forces are negligible compared to viscous forces. We are interested in the minimization of the obstacle's drag, which is given by
\begin{equation}\label{eq:drag}
F_{D} := \int_{\Gamma_{o}}-\mu (\bm{n}\cdot\nabla)\bm{v} \cdot \bm{a}+p\bm{n}\cdot\bm{a}d\Gamma,
\end{equation}
where $\bm{a}$ is defined as
\begin{align*}
\bm{a}:=(\cos(\phi),\sin(\phi))^T.
\end{align*} 
The angle of attack $\phi$ will be zero in our case. Making use of gradient-based methods, the design variable $\Gamma_o$ is modified iteratively such that the cost function $F_D$ is minimized. In order to calculate the gradient of the optimization problem \eqref{eq:optProblem}, the shape derivative of the cost function needs to be determined. As proposed in \cite{ZolesioSokolowski,DelfourZolesio2}, the shape derivative can be computed by defining a mapping $T_t[\bm{V}]$ which maps the original domain $\Omega$ to the deformed domain $\Omega_t$, given by
\begin{align*}
T_t[\bm{V}](\bm{x}) = \bm{x}+t\bm{V}(\bm{x}).
\end{align*}
The shape derivative of the cost function $F$ in the direction of the vector field $\bm{V}$ is then given by
\begin{align*}
dF(\Omega)[\bm{V}]:=\left.\frac{d}{dt} \right\vert_{t=0}F(T_t[\bm{V}](\Omega)).
\end{align*} 
To efficiently calculate the shape derivative, one often makes use of the adjoint approach, which yields the following theorem:
\begin{theorem}
The shape derivative of problem \eqref{eq:optProblem} with respect to $\bm{V}$ using the drag \eqref{eq:drag} as cost function is
\begin{align*}
dF_{D}(\bm{v},p,\Omega)[\bm{V}] = -\int_{\Gamma_{o}} (\bm{V}\cdot \bm{n})\left[ \sum_{i=1}^d \mu (\bm{n}\cdot\nabla)\lambda_i (\bm{n}\cdot\nabla)v_i \right] d\Gamma
\end{align*}
where the adjoint velocity $\bm{\lambda}\in\mathbb{R}^d$ is given by
\begin{align*}
-\mu \Delta \bm{\lambda}- \nabla \lambda_p &= \bm{0}, \\
\div{\bm{\lambda}} &= 0,\\
\bm{\lambda} = -\bm{a} &\text{  on } \Gamma_{o}
\end{align*}
and $\bm{n}$ is the normal vector of the optimization patch $\Gamma_o$.
\end{theorem}
\begin{proof}
We start by calculating the shape derivative for a general cost function
\begin{align*}
F(\bm{v},p,\Gamma_{o}) = \int_{\Gamma_{o}}f(\bm{v},(\bm{n}\cdot\nabla)\bm{v},p,\bm{n})d\Gamma.
\end{align*}
Following \cite{Schmidt10}, the shape derivative of this cost function is
\begin{align*}
dF(\Omega)[\bm{V}]=&\int_{\Gamma_o}(\bm{V}\cdot\bm{n})\left[ ( \bm{n}\cdot\nabla ) f +\kappa f\right] + \frac{\partial f}{\partial \bm{n}} d\bm{n}[\bm{V}] d\Gamma\nonumber \\
&+\int_{\Gamma_o} \frac{\partial f}{\partial v_i} v_i'[\bm{V}] + \frac{\partial f}{\partial b_i}(\bm{n}\cdot\bm{\nabla})v_i'[\bm{V}]+\frac{\partial f}{\partial p}p'[\bm{V}] d\Gamma
\end{align*}
where $\bm{b}:=(\bm{n}\cdot\nabla)\bm{v}$, the curvature is denoted by $\kappa$ and the tangential divergence of a vector field $\bm{W}$ is given by
\begin{align*}
\tandiv(\bm{W}) = \div(\bm{W}) - (\bm{n}\cdot \nabla)\bm{W} \cdot \bm{n} = \partial_{x_j}W_j-n_k \partial_{x_k}W_jn_j.
\end{align*}
Furthermore, the normal derivative $( \bm{n}\cdot\nabla )$ is only applied to the first three inputs of $f$, namely $\bm{v},\bm{b}$ and $p$. The gradient can be rewritten as
\begin{align}\label{eq:GradientBoundary}
dF(\Omega)[\bm{V}]=&\int_{\Gamma_o}(\bm{V}\cdot\bm{n})\left[ ( \bm{n}\cdot\nabla ) f +\kappa \left(f- \bm{n}\cdot \frac{\partial f}{\partial \bm{n}}\right) + \tandiv{\frac{\partial f}{\partial \bm{n}}}\right]d\Gamma \nonumber \\
&+\int_{\Gamma_o} \frac{\partial f}{\partial v_i} v_i'[\bm{V}] + \frac{\partial f}{\partial b_i}(\bm{n}\cdot\bm{\nabla})v_i'[\bm{V}]+\frac{\partial f}{\partial p}p'[\bm{V}] d\Gamma.
\end{align}
The local shape derivative of the velocity and pressure due to a perturbation $\bm{V}$ is denoted by $\bm{v}'[\bm{V}]$ and $p'[\bm{V}]$. Computing these two functions is numerically expensive, which is why we aim at finding a representation of the shape derivative, independent of these two terms. The functions $\bm{v}'[\bm{V}]$ and $p'[\bm{V}]$ are determined by linearizing the Stokes equations around the primal state $\bm{v}$ and $p$, meaning that we write down the Stokes equation for the states of the perturbed geometry with $\bm{\tilde{v}} = \bm{v}+\bm{v'}$ and $\tilde{p}=p+p'$. This yields
\begin{subequations}\label{eq:linStokes}
\begin{align}
-\mu \Delta \bm{v}'[\bm{V}] + \nabla p'[\bm{V}] &= 0,\\
\nabla \cdot \bm{v}'[\bm{V}] &= 0,\\
\bm{v}'[\bm{V}] = -(\bm{n}\cdot\nabla)\bm{v}(\bm{n}&\cdot\bm{V}) \text{ on } \Gamma_{o}. 
\end{align}
\end{subequations}
The derivation of the boundary condition can be performed by a Taylor expansion. For more details see \cite{Schmidt10}. In order to eliminate the velocity and pressure perturbations $\bm{v}'$ and $p'$ in the shape derivative \eqref{eq:GradientBoundary}, one chooses the adjoint ansatz. We start by taking the integral of the scalar product of the linearized Stokes equations \eqref{eq:linStokes} and the adjoint states $(\bm{\lambda},\lambda_p)^T$, where $\bm{\lambda}\in\mathbb{R}^d$ is the adjoint velocity and $\lambda_p$ is the adjoint pressure, leading to
\begin{equation}\label{eq:adjointPart}
0 = \int_{\Omega} -\lambda_k \mu \partial_{x_j x_j} v_k' + \lambda_k \partial_{x_k}p'+\lambda_p \partial_{x_k}v_k' d\Omega.
\end{equation}
Let us look at each term of the adjoint part individually. We start with
\begin{align*}
\int_{\Omega} -\lambda_k \mu \partial_{x_j x_j} v_k' d\Omega =& \int_{\Omega} -\mu \partial_{x_j}(\lambda_k \partial_{x_j} v_k' )+\mu\partial_{x_j}\lambda_k \partial_{x_j}v_k' d\Omega \\
=&\int_{\Gamma_o}-\mu n_j \lambda_k\partial_{x_j}v_k'd\Gamma+\int_{\Omega}\mu \partial_{x_j}\left( v_k'\partial_{x_j}\lambda_k\right)-\mu v_k'\partial_{x_j x_j}\lambda_k d\Omega\\
=&\int_{\Gamma_o}-n_j \mu \lambda_k\partial_{x_j}v_k'd\Gamma+\int_{\Gamma_o} \mu n_j v_k'\partial_{x_j}\lambda_k-\int_{\Omega}\mu v_k'\partial_{x_j x_j} \lambda_k d\Omega.
\end{align*}
Here, we used the reverse chain rule as well as Gauss divergence theorem. The remaining two terms can be transformed with the same strategy. We obtain
\begin{align*}
\int_{\Omega}\lambda_k \partial_{x_k} p' d\Omega = -\int_{\Omega} p'\partial_{x_k} \lambda_k d\Omega + \int_{\Gamma_o}n_k \lambda_k p' d\Gamma, \\ 
\int_{\Omega} \lambda_p \partial_{x_k}v_k'd\Omega = - \int_{\Omega}v_k'\partial_{x_k}\lambda_p d\Omega+\int_{\Gamma_o}\lambda_p v_k' n_k d\Gamma .
\end{align*}
Adding the transformed equation \eqref{eq:adjointPart} to the gradient \eqref{eq:GradientBoundary}, we get
\begin{align}\label{eq:GradientBoundaryAdj}
dF(\Omega)[\bm{V}]+0 =&\int_{\Gamma_o}V_l n_l\left[  n_j\partial_{x_j}  f +\kappa \left(f- n_j\frac{\partial f}{\partial n_j}\right) + \tandiv{\frac{\partial f}{\partial n_j}}\right] d\Gamma \nonumber \\
&+\int_{\Gamma_o} \frac{\partial f}{\partial v_j} v_j' + \frac{\partial f}{\partial b_j}(n_k\partial_{x_k})v_j'+\frac{\partial f}{\partial p}p' d\Gamma \nonumber \\
&+\int_{\Gamma_o}- \mu \lambda_j n_k\partial_{x_k}v_j'+\mu n_k v_j'\partial_{x_k}\lambda_j+n_j \lambda_j p'+\lambda_p v_j' n_j d\Gamma \nonumber \\ 
&+\int_{\Omega}-\mu v_k'\partial_{x_j x_j} \lambda_k - p'\partial_{x_k} \lambda_k - v_k'\partial_{x_k}\lambda_p d\Omega.
\end{align}
Remembering that the adjoint states $(\bm{\lambda},\lambda_p)^T$ are still free to choose, those states can be picked to cancel the perturbations of the primal states. The resulting constraint for the adjoint states is called the adjoint equation. By looking at the volume part of the gradient \eqref{eq:GradientBoundaryAdj}, we see that in $\Omega$ we must have
\begin{subequations}
\begin{align*}
-\mu \Delta \bm{\lambda}- \nabla \lambda_p &= \bm{0}, \\
\div{\bm{\lambda}} &= 0.
\end{align*}
\end{subequations}
Now, let us determine the adjoint boundary conditions, i.e. the conditions, which the adjoint states must fulfill on the boundaries such that the perturbed primal states drop out of the gradient \eqref{eq:GradientBoundaryAdj}. Assuming that we fulfill the adjoint equations, we can rearrange the gradient to
\begin{align*}
dF(\Omega)[\bm{V}] =&\int_{\Gamma_o}V_l n_l\left[  n_j\partial_{x_j}  f +\kappa \left(f- n_j\frac{\partial f}{\partial n_j}\right) + \tandiv{\frac{\partial f}{\partial n_j}}\right] d\Gamma \\
&+\int_{\Gamma_o} v_j'\left[ \frac{\partial f}{\partial v_j} + \mu n_k \partial_{x_k}\lambda_j +\lambda_p n_j \right] 
+(n_k\partial_{x_k})v_j' \left[ \frac{\partial f}{\partial b_j} - \mu \lambda_j \right]
+p'\left[\frac{\partial f}{\partial p}+n_j \lambda_j\right]d\Gamma.
\end{align*}
Remember that we have $v_j' = V_l n_l (n_i \partial_{x_i}) v_j$ on $\Gamma_o$ from the boundary conditions of the linearized Stokes equations \eqref{eq:linStokes}, which is why we do not need to calculate $v_j'$. The remaining perturbed primal states are forced to vanish with the help of the adjoint boundary conditions. Hence, on $\Gamma_o$ the adjoint states must fulfill
\begin{subequations}\label{eq:bcAdjoint}
\begin{align}
\frac{\partial f}{\partial \bm{b}} - \mu \bm{\lambda} = \bm{0}, \\
\frac{\partial f}{\partial p}+\bm{n}\cdot\bm{\lambda} = 0.
\end{align}
\end{subequations}
If the dual states fulfill these conditions, we are left with
\begin{align*}
dF(\Omega)[\bm{V}] =&\int_{\Gamma_o}V_l n_l\left[  n_j\partial_{x_j}  f +\kappa \left(f- n_j\frac{\partial f}{\partial n_j}\right) + \tandiv{\frac{\partial f}{\partial n_j}}\right] d\Gamma \\
&+\int_{\Gamma_o} V_l n_l (n_i \partial_{x_i}) v_j \left[ \frac{\partial f}{\partial v_j} + \mu n_k \partial_{x_k}\lambda_j +\lambda_p n_j \right]d\Gamma .
\end{align*}
Let us now simplify the gradient \eqref{eq:GradientBoundaryAdj} as well as the adjoint boundary conditions \eqref{eq:bcAdjoint} for the drag minimization problem by making use of
\begin{align*}
f_D = -\mu \bm{b} \cdot \bm{a}+p\bm{n}\cdot\bm{a}.
\end{align*}
We have
\begin{align*}
\frac{\partial f_D}{ \partial \bm{v}} &= \bm{0}\text{,}\enskip\frac{\partial f_D}{\partial \bm{b}} = -\mu \bm{a}, \\
\frac{\partial f_D}{\partial p} &= \bm{n}\cdot\bm{a},\enskip\frac{\partial f_D}{\partial \bm{n}} = -\mu \nabla \bm{v} \cdot \bm{a} + p\bm{a}.
\end{align*}
Hence, the adjoint boundary conditions on $\Gamma_{o}$ become $\bm{\lambda} = -\bm{a}$. Furthermore, the gradient changes to 
\begin{align*}
dF_D(\Omega)[\bm{V}] =\int_{\Gamma_o}V_l n_l\left[  n_j\partial_{x_j}  f_D + \tandiv{\frac{\partial f_D}{\partial n_j}} \right]+V_l n_l (n_i \partial_{x_i}) v_j \left[ \mu n_k \partial_{x_k}\lambda_j +\lambda_p n_j \right]d\Gamma,
\end{align*}
because $f_D$ is linear in the $n$ argument and consequently
\begin{align*}
f_D-\bm{n}\frac{\partial f_D}{\partial\bm{n}} = 0.
\end{align*}
Additionally, the term $V_l n_l (n_i \partial_{x_i}) v_j\lambda_p n_j$ is zero, because we can rewrite the velocity gradient as
\begin{align*}
\partial_{x_j} v_l = n_i\partial_{x_i}v_l n_j + t_k\partial_{x_k}v_l t_j.
\end{align*}
Due to the no-slip boundary condition, the derivative w.r.t. the tangential direction $\bm{t}$ drops out. If we now choose the resulting gradient to write down the mass conservation, we get
\begin{align*}
\partial_{x_j} v_j = n_i\partial_{x_i}v_j n_j = 0.
\end{align*}
Plugging in the remaining derivatives of $f_D$, we are left with
\begin{align}\label{eq:dragGrad}
dF_D(\Omega)[\bm{V}] =\int_{\Gamma_o}&V_l n_l\left[  n_j\partial_{x_j}  \left( -\mu n_k\partial_{x_k}v_i a_i+p n_k a_k \right) + \tandiv \left( -\mu \partial_{x_j} v_k a_k + p a_j \right) \right] \nonumber \\
+&V_l n_l (n_i \partial_{x_i}) v_j \mu n_k \partial_{x_k}\lambda_jd\Gamma \nonumber \\
= \int_{\Gamma_{o}}&(\bm{V}\cdot \bm{n})\left[ -\mu (\nabla_{\bm{n}})^2 \bm{v} \bm{a} + (\bm{n}\cdot\nabla)p(\bm{n}\cdot \bm{a})-\div_{\Gamma}(-\mu (\nabla \bm{v})^T \bm{a} + p\bm{a}) \right] \nonumber \\
+&(\bm{V}\cdot \bm{n})\left[ \mu (\bm{n}\cdot\nabla)\lambda_i (\bm{n}\cdot\nabla)v_i\right]d\Gamma,
\end{align}
where we have used 
\begin{align*}
(\nabla_{\bm{n}})^2 \bm{v} := n_j\partial_{x_j}\left( n_k\partial_{x_k}v_i\right).
\end{align*}
The derived shape derivative can further be simplified to facilitate the derivation of the Hessian: Taking a closer look at the tangential divergence part of \eqref{eq:dragGrad}, one sees that the term inside the tangential divergence becomes
\begin{align*}
\div_{\Gamma}(-\mu \nabla \bm{v}\cdot\bm{a}) &= \div(-\mu \nabla \bm{v}\cdot\bm{a}) - \bm{n}\cdot \nabla(-\mu \nabla \bm{v}\cdot\bm{a}) \cdot\bm{n} \\
&=-\partial_{x_j}(\mu \partial_{x_j} v_k a_k )+\mu n_i\partial_{x_i}(\partial_{x_j} v_k a_k)n_j\\
&=-\mu\partial_{x_j x_j}v_k a_k+\mu n_i\partial_{x_i x_j} v_k a_k  n_j.
\end{align*}
For the remaining term, we get
\begin{align*}
\div_{\Gamma}(p\bm{a}) &= \div(p\bm{a}) - (\bm{n}\cdot \nabla)(p\bm{a}) \cdot\bm{n} \\
&=\partial_{x_j} p a_j -n_k\partial_{x_k}p a_j n_j.
\end{align*}
Hence, the tangential divergence term in \eqref{eq:dragGrad} becomes
\begin{align*}
&\int_{\Gamma_{o}} (V_l n_l)\left[ -\mu\partial_{x_j x_j}v_k a_k+\mu n_i\partial_{x_i x_j} v_k a_k  n_j+  \partial_{x_j} p a_j -n_k\partial_{x_k}p a_j n_j \right]d\Gamma \nonumber \\
=& \int_{\Gamma_{o}} (\bm{V}\cdot \bm{n})\left[ -\mu\Delta \bm{v} \cdot \bm{a}+\mu (\nabla_{\bm{n}})^2 \bm{v} \bm{a}+  \nabla p \cdot \bm{a} -(\bm{n}\cdot\nabla)p(\bm{n}\cdot \bm{a}) \right] d\Gamma \nonumber \\
=& \int_{\Gamma_{o}} (\bm{V}\cdot \bm{n})\left[ (-\mu\Delta \bm{v} +  \nabla p) \cdot \bm{a}+\mu (\nabla_{\bm{n}})^2 \bm{v} \bm{a}-(\bm{n}\cdot\nabla)p(\bm{n}\cdot \bm{a}) \right]d\Gamma \nonumber \\
=& \int_{\Gamma_{o}} (\bm{V}\cdot \bm{n})\left[ \mu (\nabla_{\bm{n}})^2 \bm{v} \bm{a}-(\bm{n}\cdot\nabla)p(\bm{n}\cdot \bm{a}) \right]d\Gamma.
\end{align*}
Note, that $-\mu\Delta \bm{v} +  \nabla p$ is zero, due to the fact that the state variables fulfill the Stokes equations. Now most of the terms in \eqref{eq:dragGrad} cancel, meaning that we are left with
\begin{align*}
dF_{D}(\Omega)[\bm{V}] = \int_{\Gamma_{o}} (\bm{V}\cdot \bm{n}) df_D d\Gamma
\end{align*}
where 
\begin{align}\label{eq:df}
df_D:=-\sum_{i=1}^d \mu (\bm{n}\cdot\nabla)\lambda_i (\bm{n}\cdot\nabla)v_i.
\end{align}
\qed
\end{proof}
With the help of the adjoint approach, a numerically cheap calculation of the shape derivative can be ensured, as the computational costs no longer depend on the number of design parameters. This motivates using a detailed description of the optimization patch $\Gamma_o$ by using the nodes of the discretized surface, defined by $\bm{x}_k$ for $k = 1,...,N$ as design parameters. Having derived the shape derivative of the optimization problem \eqref{eq:optProblem}, one can iteratively approach the optimal design with a steepest-descent update. To obtain a search direction for every surface node with the help of the shape derivative, the perturbations
\begin{align}\label{eq:Vk}
\bm{V}_k(\bm{x}) := \bm{n}(\bm{x})\varphi_k(\bm{x}),
\end{align}
for $k = 1,...,N$ is defined, where $\varphi_k:\Gamma_o\to\mathbb{R}$ are piece-wise linear basis functions fulfilling $\varphi_k(\bm{x}_l) = 1 \text{ if }\bm{x}_l = \bm{x}_k$.
The deformation of the $k$-{th} mesh node is now given by
\begin{align*}
\bm{x}_k^{\text{new}} &= \bm{x}_k - \bm{V}_k(\bm{x}_k) dF_{D}[\bm{V}_k(\bm{x}_k)] \\
&= \bm{x}_k - \bm{n}(\bm{x}_k)\varphi_k(\bm{x}_k)dF_{D}[\bm{n}(\bm{x}_k)\varphi_k(\bm{x}_k)] \\
&\approx \bm{x}_k - \frac{1}{2}\bm{n}(\bm{x}_k)df_D(\bm{x}_k)\left( \Vert \bm{x}_k-\bm{x}_{k-1}\Vert + \Vert \bm{x}_{k+1}-\bm{x}_{k}\Vert \right),
\end{align*}
where we used a first order quadrature rule to evaluate the integral in $dF_{D}$. Choosing an adequate step size $\gamma$ yields the steepest-descent update
\begin{align*}
\bm{x}_k^{(l+1)} = \bm{x}_k^{(l)} +\gamma  p_k^{(l)} \bm{n}\left(x_k^{(l)}\right),
\end{align*}
where the steepest-descent search direction is given by
\begin{align*}
p_k^{(l)} = -df_D\left(\bm{x}_k^{(l)}\right).
\end{align*}
Alternatively, we can collect all values of the gradient evaluated at the surface points in a vector 
\begin{align}\label{eq:GradientVector}
\bm{df}^{(l)} = \left( df_D\left(\bm{x}_1^{(l)}\right), \cdots, df_D\left(\bm{x}_{N}^{(l)}\right) \right)^T,
\end{align}
yielding the steepest-descent search direction
\begin{align*}
\bm{p}^{(l)} = -\bm{df}^{(l)}.
\end{align*}
As already discussed, the convergence of steepest-descent is slow. Additionally, the gradient $df_D$ is of insufficient regularity, leading to rough designs with subsequent problems in getting the flow solver to converge. To overcome this problem, we derive the Hessian of the optimization problem analytically. When given a Hessian matrix $\bm{H}\in\mathbb{R}^{N \times N}$, we can choose the Newton search direction
\begin{align}\label{eq:NewtonDirection}
\bm{p}^{(l)} = -\bm{H}^{-1} \bm{df}^{(l)}.
\end{align}
The derivation of the Hessian will show that the inverse Hessian will have properties of a smoothing method, which is why we can combine the tasks of accelerating the optimization and smoothing the search direction.

\section{The analytic Hessian symbol}
\label{sec:section3}
In order to accelerate the optimization process, we wish to make use of Hessian information, or to be more precise, the symbol of the Hessian. This derivation uses the techniques introduced in \cite{arian1995analysis,arian1999analysis,arian1999preconditioning,schmidt2009impulse}. In contrast to previous works, our analysis holds for smooth geometries beyond the typical upper half-plane, allowing the derivation of the Hessian symbol in applications of practical interest. For a given operator $L$, its symbol $\sigma_L$ is the response of $L$ to a wave with a fixed frequency $\omega$. To give a brief understanding of operator symbols, we look at the following example:
\begin{example}
We derive the symbol of the operator
\begin{align*}
Lg := \left(1-\frac{d^2}{dx^2}\right)g.
\end{align*}
To derive the response of $L$ to an input wave, we choose $g = e^{-i\omega x}$, which yields
\begin{align*}
L e^{-i\omega x} = \left(1+\omega^2\right)e^{-i\omega x}.
\end{align*}
The symbol is therefore given by $\sigma_L = 1+\omega^2$. It can be seen that the operator $L$ amplifies frequencies quadratically with respect to the input frequency $\omega$. The fact that $\sigma_L$ is a real number, tells us that the operator does not cause a phase shift.
\end{example}
Our aim is to derive the Hessian symbol $\sigma_H$ for the drag minimization problem when using the Stokes equations. For this, the Hessian response to a Fourier mode with frequency $\omega$, which is used to perturb the optimization patch $\Gamma_{o}$ is investigated analytically. We assume a two-dimensional geometry, which can be described by body fitted coordinates $\xi_1$ and $\xi_2$. A mapping to the physical coordinates is given by
\begin{align*}
\Phi\left(\xi_1,\xi_2\right) = \bm{x}.
\end{align*}
The physical coordinates of the optimization patch are
\begin{align*}
\Gamma_o(\xi_1) = \Phi\left(\xi_1,0\right),
\end{align*}
meaning that $\xi_1$ is the parameter describing the position on the optimization patch. We choose the parametrization such that
\begin{align*}
\left\Vert \frac{d}{d \xi_1}\Gamma_o(\xi_1)\right\Vert = 1,
\end{align*} 
i.e., the tangential vector $\bm{t}$ has unit length. If the remaining parameter $\xi_2$ is used as a parametrization into the normal direction $\bm{n}$, we can write our mapping as
\begin{align*}
\Phi\left(\xi_1,\xi_2\right) = \Gamma_o(\xi_1)+\xi_2\bm{n}( \Gamma_o(\xi_1) ).
\end{align*}
In this setting, we derive the symbol of the Hessian in the following theorem.
\begin{theorem}
The symbol of the Hessian for the Stokes equations is given by
\begin{align}\label{eq:HessianSymbol}
\sigma_H = \beta_1 + \beta_2 \omega 
\end{align}
where
\begin{align}\label{eq:beta1}
\beta_1 = \mu (\bm{n}\cdot\nabla) \left(\sum_{k = 1}^2 (\bm{n}\cdot\nabla)v_k (\bm{n}\cdot\nabla)\lambda_k \right) 
\end{align}
and
\begin{equation}\label{eq:beta2}
\beta_2 = - 2\mu \sum_{k = 1}^2 (\bm{n}\cdot\nabla) \lambda_k (\bm{n}\cdot\nabla) v_k .
\end{equation}
\end{theorem}
\begin{proof}
The Hessian of our problem is the response of the gradient $df_D$ given in \eqref{eq:df} to a perturbation of the design space, which we call $\alpha$. The response of a function $g$ due to a perturbation $\alpha$ is denoted as
\begin{align*}
g'[\alpha] = \lim_{\epsilon\rightarrow 0} \frac{g(\Gamma_o^\epsilon)- g(\Gamma_o)}{\epsilon},
\end{align*}
where the perturbed surface is given by
\begin{align*}
\Gamma_o^\epsilon := \Gamma_o + \epsilon \alpha\bm{n}.
\end{align*}
The response of the gradient $df_D$ to such a perturbation is now given by
\begin{equation}\label{eq:gradientPert}
df_D'[\alpha] = -\mu n_k\partial_{x_k}\lambda_i'[\alpha] n_l\partial_{x_l}v_i-\mu n_k\partial_{x_k}\lambda_i n_l\partial_{x_l}v_i'[\alpha].
\end{equation}
The change of the state variables as well as the adjoint variables due to a small perturbation $\alpha$ in the normal direction can be computed from the linearized primal and adjoint state equations, which are 
\begin{align*}
-\mu \Delta \bm{v}'[\alpha] + \nabla p'[\alpha] &= 0,\\
\nabla \cdot \bm{v}'[\alpha] &= 0,\\
\bm{v}'[\alpha] = -(\bm{n}\cdot\nabla)\bm{v} \alpha \text{ on } &\Gamma_{o},
\end{align*}
and
\begin{align*}
-\mu \Delta \bm{\lambda}'[\alpha] - \nabla \lambda_p'[\alpha] &= 0,\\
\nabla \cdot \bm{\lambda}'[\alpha] &= 0,\\
\bm{\lambda}'[\alpha] = -(\bm{n}\cdot\nabla)\bm{\lambda}\alpha \text{ on } &\Gamma_{o}.
\end{align*}
Transforming these equations into body fitted coordinates $(\xi_1,\xi_2)$ leads to
\begin{subequations}\label{eq:primalPert}
\begin{align}
-\mu \frac{\partial \xi_l}{\partial x_i}\frac{\partial^2 v_j'[\alpha]}{\partial \xi_k \xi_l}\frac{\partial \xi_k}{\partial x_i}  -\mu \frac{\partial v_j'}{\partial\xi_k}\frac{\partial^2 \xi_k}{\partial x_i^2} + \frac{\partial p'[\alpha]}{\partial \xi_k}\frac{\partial \xi_k}{\partial x_j}  &= 0,\\
\frac{\partial v_i'[\alpha]}{\partial \xi_k}\frac{\partial \xi_k}{\partial x_i} &= 0,\\
v_j'[\alpha] = -n_k \frac{\partial v_j}{\partial \xi_i}\frac{\partial \xi_i}{\partial x_k} \alpha \text{ on } &\Gamma_{o},
\end{align}
\end{subequations}
and
\begin{subequations}\label{eq:adjointPert}
\begin{align}
-\mu \frac{\partial \xi_l}{\partial x_i}\frac{\partial^2 \lambda_j'[\alpha]}{\partial \xi_k \xi_l}\frac{\partial \xi_k}{\partial x_i}-\mu\frac{\partial \lambda_j'}{\partial\xi_k}\frac{\partial^2 \xi_k}{\partial x_i^2}  - \frac{\partial \lambda_p'[\alpha]}{\partial \xi_k}\frac{\partial \xi_k}{\partial x_j}  &= 0,\\
\frac{\partial \lambda_i'[\alpha]}{\partial \xi_k}\frac{\partial \xi_k}{\partial x_i} &= 0,\\
\lambda_j'[\alpha] = -n_k \frac{\partial \lambda_j}{\partial \xi_i}\frac{\partial \xi_i}{\partial x_k} \alpha \text{ on } &\Gamma_{o}.
\end{align}
\end{subequations}
As our goal is to determine the Hessian response to a Fourier mode, we let $\alpha$ be a mode with frequency $\omega_1$, meaning that we have
\begin{align*}
\alpha = e^{i \omega_1 \xi_1}.
\end{align*}
Furthermore, we make the assumption that the perturbed states have the form
\begin{align}\label{eq:assumptionPerturbed}
\bm{v}'[\alpha] &= \hat{\bm{v}} e^{i \omega_1 \xi_1} e^{i \omega_2^{p} \xi_2},\enskip
p'[\alpha] = \hat{p} e^{i \omega_1 \xi_1} e^{i \omega_2^{p} \xi_2},\nonumber \\
\bm{\lambda}'[\alpha] &= \bm{\hat{\lambda}} e^{i \omega_1 \xi_1} e^{i \omega_2^{a} \xi_2},\enskip
\lambda_p'[\alpha] = \hat{\lambda}_p e^{i \omega_1 \xi_1} e^{i \omega_2^{a} \xi_2}.
\end{align}
It is important to note that the choice of the dependency on $\xi_1$ is straight forward, since we would like to match the boundary conditions of the perturbed state variables for $\xi_2 = 0$. The complex exponential or wave like dependency in $\xi_2$ direction is a Fourier ansatz. There are two unknowns that need to be determined, namely the amplitudes which are the $\hat{\bullet}$ variables as well as the response frequencies $\omega_2^{p,a}$. Our first goal is to determine these response frequencies $\omega_2^{p}$ and $\omega_2^{a}$. For this, we insert our ansatz for the perturbed states into the linearized state equations \eqref{eq:primalPert} and \eqref{eq:adjointPert}. For $\omega_2 = \omega_2^{p}$ we now must fulfill
\begin{align*}
\begin{pmatrix}
\mu \omega_l\omega_k\frac{\partial \xi_l}{\partial x_j}\frac{\partial \xi_k}{\partial x_j} -i\omega_k \mu\frac{\partial^2 \xi_k}{\partial x_j^2}& 0 & i \omega_k \frac{\partial \xi_k}{\partial x_1} \\
0 & \mu \omega_l \omega_k\frac{\partial \xi_l}{\partial x_j}\frac{\partial \xi_k}{\partial x_j}-i\omega_k\mu \frac{\partial^2 \xi_k}{\partial x_j^2} & i \omega_k \frac{\partial \xi_k}{\partial x_2} \\
i \omega_k \frac{\partial \xi_k}{\partial x_1} & i \omega_k \frac{\partial \xi_k}{\partial x_2} & 0  \\
\end{pmatrix}
\begin{pmatrix}
\hat{v}_1 \\
\hat{v}_2 \\
\hat{p}  \\
\end{pmatrix}
=
\begin{pmatrix}
0 \\
0 \\
0  \\
\end{pmatrix}
\end{align*}
as well as for $\omega_2 = \omega_2^{a}$ 
\begin{align*}
\begin{pmatrix}
\mu \omega_l\omega_k\frac{\partial \xi_l}{\partial x_j}\frac{\partial \xi_k}{\partial x_j}-i\omega_k \mu\frac{\partial^2 \xi_k}{\partial x_j^2}& 0 & -i \omega_k \frac{\partial \xi_k}{\partial x_1} \\
0 & \mu \omega_l \omega_k\frac{\partial \xi_l}{\partial x_j}\frac{\partial \xi_k}{\partial x_j}-i\omega_k \mu\frac{\partial^2 \xi_k}{\partial x_j^2} & -i \omega_k \frac{\partial \xi_k}{\partial x_2} \\
i \omega_k \frac{\partial \xi_k}{\partial x_1} & i \omega_k \frac{\partial \xi_k}{\partial x_2} & 0  \\
\end{pmatrix}
\begin{pmatrix}
\hat{\lambda}_1 \\
\hat{\lambda}_2 \\
\hat{\lambda}_p  \\
\end{pmatrix}
=
\begin{pmatrix}
0 \\
0 \\
0  \\
\end{pmatrix}.
\end{align*}

These two systems of equations only have a non-trivial solution if the determinant of the two matrices is zero. Note that these two matrices only differ in the sign of the last row, leading to determinants, which have the same roots. Therefore, every non-trivial response frequency of the primal system is also a valid response frequency of the adjoint system. Hence, we denote $\omega_2^{p}$ and $\omega_2^{a}$ as $\omega_2$, which leads to the determinant
\begin{align*}
&\left(\sum_{l,k,j}\mu \omega_l\omega_k\frac{\partial \xi_l}{\partial x_j}\frac{\partial \xi_k}{\partial x_j}-\sum_{k,j}i\omega_k\mu \frac{\partial^2 \xi_k}{\partial x_j^2}\right)\left( \sum_{k}\omega_k \frac{\partial \xi_k}{\partial x_2} \sum_{k} \omega_k \frac{\partial \xi_k}{\partial x_2}\right)\\
&+\left(\sum_k\omega_k \frac{\partial \xi_k}{\partial x_1}\right)\left( \sum_{k}\omega_k \frac{\partial \xi_k}{\partial x_1}  \left(\sum_{l,k,j} \mu \omega_l \omega_k\frac{\partial \xi_l}{\partial x_j}\frac{\partial \xi_k}{\partial x_j}-\sum_{k,j}i\omega_k\mu \frac{\partial^2 \xi_k}{\partial x_j^2}\right) \right) \\
=&\left(\sum_{l,k,j}\mu \omega_l\omega_k\frac{\partial \xi_l}{\partial x_j}\frac{\partial \xi_k}{\partial x_j}-\sum_{k,j}i\omega_k \mu\frac{\partial^2 \xi_k}{\partial x_j^2}\right)\left[ \left( \sum_{k}\omega_k \frac{\partial \xi_k}{\partial x_1}\right)^2 + \left( \sum_{k}\omega_k \frac{\partial \xi_k}{\partial x_2}\right)^2 \right] \\
=&\mu\left(\left( \sum_{k}\omega_k \frac{\partial \xi_k}{\partial x_1}\right)^2 + \left( \sum_{k}\omega_k \frac{\partial \xi_k}{\partial x_2}\right)^2-\sum_{k,j}i\omega_k \frac{\partial^2 \xi_k}{\partial x_j^2}\right)\left[ \left( \sum_{k}\omega_k \frac{\partial \xi_k}{\partial x_1}\right)^2 + \left( \sum_{k}\omega_k \frac{\partial \xi_k}{\partial x_2}\right)^2 \right]\stackrel{!}{=} 0.
\end{align*}
Here, we no longer use Einstein's sum convention such that we can reuse indices. Let us determine the roots of the polynomial inside the square brackets to obtain a valid response frequency $\omega_2$. The determinant will be zero if $\omega_2$ fulfills
\begin{align*}
&\left( \sum_{k}\omega_k \frac{\partial \xi_k}{\partial x_1}\right)^2 + \left( \sum_{k}\omega_k \frac{\partial \xi_k}{\partial x_2}\right)^2 = 0\\
\Leftrightarrow& \omega_1^2\left(\frac{\partial \xi_1}{\partial x_2}\right)^2+2\omega_1\omega_2\frac{\partial \xi_1}{\partial x_2}\frac{\partial \xi_2}{\partial x_2}+\omega_2^2\left(\frac{\partial \xi_2}{\partial x_2}\right)^2 +\omega_1^2\left(\frac{\partial \xi_1}{\partial x_1}\right)^2+2\omega_1\omega_2\frac{\partial \xi_1}{\partial x_1}\frac{\partial \xi_2}{\partial x_1}+\omega_2^2\left(\frac{\partial \xi_2}{\partial x_1}\right)^2 = 0 \\
\Leftrightarrow&\omega_2^2\left(\left(\frac{\partial \xi_2}{\partial x_2}\right)^2+\left(\frac{\partial \xi_2}{\partial x_1}\right)^2\right)+\omega_2\left[ 2\omega_1\left(\frac{\partial \xi_1}{\partial x_2}\frac{\partial \xi_2}{\partial x_2}+\frac{\partial \xi_1}{\partial x_1}\frac{\partial \xi_2}{\partial x_1}\right)\right] + \omega_1^2\left(\frac{\partial \xi_1}{\partial x_2}\right)^2+\omega_1^2\left(\frac{\partial \xi_1}{\partial x_1}\right)^2 =0.
\end{align*}
For simplicity, we define 
\begin{align*}
c:=\frac{1}{\left(\frac{\partial \xi_2}{\partial x_2}\right)^2+\left(\frac{\partial \xi_2}{\partial x_1}\right)^2}.
\end{align*}
Applying the $p,q-$formula with 
\begin{align*}
p &= 2c\omega_1\left( \frac{\partial \xi_1}{\partial x_2}\frac{\partial \xi_2}{\partial x_2}+\frac{\partial \xi_1}{\partial x_1}\frac{\partial \xi_2}{\partial x_1} \right) \\
q &= c \omega_1^2\left[\left(\frac{\partial \xi_1}{\partial x_2}\right)^2+\left(\frac{\partial \xi_1}{\partial x_1}\right)^2\right]
\end{align*}
yields
\begin{align*}
\omega_2^{1,2} =& -c\omega_1\left( \frac{\partial \xi_1}{\partial x_2}\frac{\partial \xi_2}{\partial x_2}+\frac{\partial \xi_1}{\partial x_1}\frac{\partial \xi_2}{\partial x_1} \right) \pm \omega_1 \sqrt{ c^2 \left(\frac{\partial \xi_1}{\partial x_2}\frac{\partial \xi_2}{\partial x_2}+\frac{\partial \xi_1}{\partial x_1}\frac{\partial \xi_2}{\partial x_1} \right)^2- c\left[\left(\frac{\partial \xi_1}{\partial x_2}\right)^2+\left(\frac{\partial \xi_1}{\partial x_1}\right)^2\right] } \\
=& -c\omega_1 \left[ \left( \frac{\partial \xi_1}{\partial x_2}\frac{\partial \xi_2}{\partial x_2}+\frac{\partial \xi_1}{\partial x_1}\frac{\partial \xi_2}{\partial x_1} \right) \mp  \sqrt{ \left(\frac{\partial \xi_1}{\partial x_2}\frac{\partial \xi_2}{\partial x_2}+\frac{\partial \xi_1}{\partial x_1}\frac{\partial \xi_2}{\partial x_1}\right)^2 - \frac{1}{c}\left[\left(\frac{\partial \xi_1}{\partial x_2}\right)^2+\left(\frac{\partial \xi_1}{\partial x_1}\right)^2\right] }\right].
\end{align*}
Let us take a closer look at the term inside the square root, which is
\begin{align*}
&\left(\frac{\partial \xi_1}{\partial x_2}\frac{\partial \xi_2}{\partial x_2}+\frac{\partial \xi_1}{\partial x_1}\frac{\partial \xi_2}{\partial x_1}\right)^2 - \frac{1}{c}\left[\left(\frac{\partial \xi_1}{\partial x_2}\right)^2+\left(\frac{\partial \xi_1}{\partial x_1}\right)^2\right]\\
=&\left(\frac{\partial \xi_1}{\partial x_2}\frac{\partial \xi_2}{\partial x_2}\right)^2+2\frac{\partial \xi_1}{\partial x_2}\frac{\partial \xi_2}{\partial x_2}\frac{\partial \xi_1}{\partial x_1}\frac{\partial \xi_2}{\partial x_1}+\left(\frac{\partial \xi_1}{\partial x_1}\frac{\partial \xi_2}{\partial x_1}\right)^2 - \left(\left(\frac{\partial \xi_1}{\partial x_2}\right)^2+\left(\frac{\partial \xi_1}{\partial x_1}\right)^2\right)\left(\left(\frac{\partial \xi_2}{\partial x_2}\right)^2+\left(\frac{\partial \xi_2}{\partial x_1}\right)^2\right)\\
=&2\frac{\partial \xi_1}{\partial x_2}\frac{\partial \xi_2}{\partial x_2}\frac{\partial \xi_1}{\partial x_1}\frac{\partial \xi_2}{\partial x_1} -\left(\frac{\partial \xi_1}{\partial x_2}\right)^2\left(\frac{\partial \xi_2}{\partial x_1}\right)^2-\left(\frac{\partial \xi_1}{\partial x_1}\right)^2\left(\frac{\partial \xi_2}{\partial x_2}\right)^2 \\
=&-\left( \frac{\partial \xi_1}{\partial x_2}\frac{\partial \xi_2}{\partial x_1}-\frac{\partial \xi_1}{\partial x_1}\frac{\partial \xi_2}{\partial x_2} \right)^2.
\end{align*}
Since this term is always negative, we know that the square root will result in a complex term, leading to
\begin{equation}\label{eq:omega1234}
\omega_2^{1,2} =-c\omega_1 \left[ \left( \frac{\partial \xi_1}{\partial x_2}\frac{\partial \xi_2}{\partial x_2}+\frac{\partial \xi_1}{\partial x_1}\frac{\partial \xi_2}{\partial x_1} \right) \mp  i\left( \frac{\partial \xi_1}{\partial x_2}\frac{\partial \xi_2}{\partial x_1}-\frac{\partial \xi_1}{\partial x_1}\frac{\partial \xi_2}{\partial x_2} \right)\right].
\end{equation}
This means the system can be solved for $\omega_2^{1,2}(\omega_1)$. We express the derivatives of the body fitted coordinates as derivatives of the physical coordinates, which have an intuitive geometric meaning on the boundary, since
\begin{align*}
&\left(\frac{\partial x_1}{\partial\xi_1},\frac{\partial x_2}{\partial\xi_1}\right)^T_{\xi_2 = 0} = \frac{d}{d \xi_1}\Gamma_o(\xi_1) = \bm{t}(\xi_1), \\
&\left(\frac{\partial x_1}{\partial\xi_2},\frac{\partial x_2}{\partial\xi_2}\right)^T_{\xi_2 = 0} = \bm{n}\left(\Gamma_o(\xi_1)\right).
\end{align*}
The relation between the derivatives of the two coordinate systems can be determined by noting that
\begin{align*}
\begin{pmatrix}
\frac{\partial }{\partial \xi_1} \\
\frac{\partial }{\partial \xi_2} \\
\end{pmatrix}
=
\begin{pmatrix}
\frac{\partial x_1}{\partial \xi_1}& \frac{\partial x_2}{\partial \xi_1}  \\
\frac{\partial x_1}{\partial \xi_2} & \frac{\partial x_2}{\partial \xi_2} \\
\end{pmatrix}
\begin{pmatrix}
\frac{\partial }{\partial x_1} \\
\frac{\partial }{\partial x_2} \\
\end{pmatrix}
\end{align*}
and
\begin{align*}
\begin{pmatrix}
\frac{\partial }{\partial x_1} \\
\frac{\partial }{\partial x_2} \\
\end{pmatrix}
=
\begin{pmatrix}
\frac{\partial \xi_1}{\partial x_1}& \frac{\partial \xi_2}{\partial x_1}  \\
\frac{\partial \xi_1}{\partial x_2} & \frac{\partial \xi_2}{\partial x_2} \\
\end{pmatrix}
\begin{pmatrix}
\frac{\partial }{\partial \xi_1} \\
\frac{\partial }{\partial \xi_2} \\
\end{pmatrix}
.
\end{align*}
This means that
\begin{align*}
\begin{pmatrix}
\frac{\partial \xi_1}{\partial x_1}& \frac{\partial \xi_2}{\partial x_1}  \\
\frac{\partial \xi_1}{\partial x_2} & \frac{\partial \xi_2}{\partial x_2} \\
\end{pmatrix}
=
\begin{pmatrix}
\frac{\partial x_1}{\partial \xi_1}& \frac{\partial x_2}{\partial \xi_1}  \\
\frac{\partial x_1}{\partial \xi_2} & \frac{\partial x_2}{\partial \xi_2} \\
\end{pmatrix}^{-1}
=
\frac{1}{\frac{\partial x_1}{\partial \xi_1}\frac{\partial x_2}{\partial \xi_2}-\frac{\partial x_2}{\partial \xi_1}\frac{\partial x_1}{\partial \xi_2}}
\begin{pmatrix}
\frac{\partial x_2}{\partial \xi_2}& -\frac{\partial x_2}{\partial \xi_1}  \\
-\frac{\partial x_1}{\partial \xi_2} & \frac{\partial x_1}{\partial \xi_1} \\
\end{pmatrix}
.
\end{align*}
The response frequency \eqref{eq:omega1234} can now be evaluated on the boundary:
\begin{align*}
\left.\omega_2^{1,2}\right\vert_{\Gamma_o} =&\left.\mp \frac{c\omega_1}{\frac{\partial x_1}{\partial \xi_1}\frac{\partial x_2}{\partial \xi_2}-\frac{\partial x_2}{\partial \xi_1}\frac{\partial x_1}{\partial \xi_2}} \left[ -\frac{\partial x_1}{\partial \xi_2}\frac{\partial x_1}{\partial \xi_1}-\frac{\partial x_2}{\partial \xi_2}\frac{\partial x_2}{\partial \xi_1}  
\mp  i\left( \frac{\partial x_1}{\partial \xi_2}\frac{\partial x_2}{\partial \xi_1}-\frac{\partial x_2}{\partial \xi_2}\frac{\partial x_1}{\partial \xi_1} \right)\right]\right\vert_{\Gamma_o} \\
=&\left.\mp \frac{c\omega_1}{\frac{\partial x_1}{\partial \xi_1}\frac{\partial x_2}{\partial \xi_2}-\frac{\partial x_2}{\partial \xi_1}\frac{\partial x_1}{\partial \xi_2}} \left[ -\bm{n}\cdot\bm{t} \mp  i\left( \frac{\partial x_1}{\partial \xi_2}\frac{\partial x_2}{\partial \xi_1}-\frac{\partial x_2}{\partial \xi_2}\frac{\partial x_1}{\partial \xi_1} \right)\right]\right\vert_{\Gamma_o} \\
=& \pm i \left.c\right\vert_{\Gamma_o}\omega_1.
\end{align*}
For $c$ we obtain
\begin{align*}
c\vert_{\Gamma_o} =& \frac{1}{\left(\frac{\partial x_1}{\partial \xi_1}\frac{\partial x_2}{\partial \xi_2}-\frac{\partial x_2}{\partial \xi_1}\frac{\partial x_1}{\partial \xi_2}\right)^2}\frac{1}{\left(\frac{\partial x_1}{\partial \xi_1}\right)^2+\left(\frac{\partial x_2}{\partial \xi_1}\right)^2}=\frac{1}{\Vert \bm{t} \Vert^6 } = 1,
\end{align*}
due to the fact that $\bm{\hat{t}}:=\left(\partial_{\xi_2}x_2,-\partial_{\xi_2}x_1\right)^T$ is either $\bm{t}$ or $-\bm{t}$, since
\begin{align*}
\bm{\hat{t}}\cdot \bm{n} = 0, \enskip \left\Vert \bm{\hat{t}} \right\Vert = 1.
\end{align*}
We now have two possible choices for $\omega_2^{p}$ and $\omega_2^{a}$, which will allow a non-trivial solution, namely
\begin{equation}\label{eq:omega12OnGamma}
\left.\omega_2^{1,2}\right\vert_{\Gamma_o} = \pm i \omega_1.
\end{equation} 
Inserting the expression for $\omega_2$, which we have derived in \eqref{eq:omega1234}, into the assumption for the perturbed state variables \eqref{eq:assumptionPerturbed}, we get
\begin{align*}
\bm{v}'[\alpha] &= \hat{\bm{v}} e^{i \omega_1 \xi_1} e^{i\omega_2^{1,2}(\omega_1)\xi_2},\enskip
p'[\alpha] = \hat{p} e^{i \omega_1 \xi_1} e^{i\omega_2^{1,2}(\omega_1)\xi_2},\\
\bm{\lambda}'[\alpha] &= \bm{\hat{\lambda}} e^{i \omega_1 \xi_1} e^{i\omega_2^{1,2}(\omega_1)\xi_2},\enskip
\lambda_p'[\alpha] = \hat{\lambda}_p e^{i \omega_1 \xi_1} e^{i\omega_2^{1,2}(\omega_1)\xi_2}.
\end{align*}
The remaining unknowns in our ansatz for the perturbed primal and adjoint states are the $\hat{\bullet}$ variables, which can be determined with the help of the boundary conditions. Remember, that we are only interested in knowing the perturbed states, which influence the perturbation of the gradient \eqref{eq:gradientPert}, namely $\bm{v}'[\alpha]$ and $\bm{\lambda}'[\alpha]$. The boundary conditions contain normal derivatives of those states, which is why we first write down the normal derivatives for boundary fitted coordinates. We have
\begin{align*}
n_k \frac{\partial W_j}{\partial \xi_i}\frac{\partial \xi_i}{\partial x_k} =& \frac{1}{\frac{\partial x_1}{\partial \xi_1}\frac{\partial x_2}{\partial \xi_2}-\frac{\partial x_2}{\partial \xi_1}\frac{\partial x_1}{\partial \xi_2}}\left( \left( \frac{\partial x_2}{\partial \xi_2}, -\frac{\partial x_1}{\partial \xi_2}\right)^T \cdot\bm{n} \frac{\partial  W_j}{\partial \xi_1} +\left( -\frac{\partial x_2}{\partial \xi_1}, \frac{\partial x_1}{\partial \xi_1}\right)^T \cdot\bm{n} \frac{\partial  W_j}{\partial \xi_2} \right)\\
=& \frac{1}{\frac{\partial x_1}{\partial \xi_1}\frac{\partial x_2}{\partial \xi_2}-\frac{\partial x_2}{\partial \xi_1}\frac{\partial x_1}{\partial \xi_2}}\left( \left( n_2, -n_1\right)^T \cdot\bm{n} \frac{\partial  W_j}{\partial \xi_1} +\left( -\frac{\partial x_2}{\partial \xi_1}, \frac{\partial x_1}{\partial \xi_1}\right)^T \cdot\left( \frac{\partial x_1}{\partial \xi_2}, \frac{\partial x_2}{\partial \xi_2}\right)^T \frac{\partial  W_j}{\partial \xi_2} \right)\\
=&\frac{\partial  W_j}{\partial \xi_2}.
\end{align*}
Plugging this expression into the boundary condition of the perturbed state variables given in \eqref{eq:primalPert} and \eqref{eq:adjointPert} leads to
\begin{align*}
v_j'[\alpha] &= \hat{v}_j e^{i \omega_1 \xi_1} = - \frac{\partial v_j}{\partial \xi_2}\alpha,\\
\lambda_j'[\alpha] &= \hat{\lambda}_j e^{i \omega_1 \xi_1} =- \frac{\partial \lambda_j}{\partial \xi_2} \alpha,
\end{align*}
meaning that
\begin{align*}
\hat{\bm{v}} &= - \partial_{\xi_2}\bm{v},\enskip \hat{\bm{\lambda}} = - \partial_{\xi_2}\bm{\lambda}.
\end{align*}
If we now write down the ansatz for the perturbed state variables, we get
\begin{align*}
\bm{v}'[\alpha] = - \partial_{\xi_2}\bm{v} e^{i \omega_1 \xi_1} e^{i\omega_2^{1,2}(\omega_1)\xi_2},\enskip \bm{\lambda}'[\alpha] = - \partial_{\xi_2}\bm{\lambda} e^{i \omega_1 \xi_1} e^{i\omega_2^{1,2}(\omega_1)\xi_2}.
\end{align*}
Let us now use this solution to calculate the unknown terms in the perturbed gradient \eqref{eq:gradientPert} on the boundary $\Gamma_o$, meaning that $\xi_2=0$. The perturbed gradient for boundary fitted coordinates is given by
\begin{align*}
df_D'[\alpha] = -\mu \partial_{\xi_2}\lambda_k'[\alpha] \partial_{\xi_2}v_k-\mu \partial_{\xi_2}\lambda_k \partial_{\xi_2}v_k'[\alpha].
\end{align*}
On the boundary, where $\xi_2=0$, we have that 
\begin{align*}
\partial_{\xi_2}v_k'[\alpha] = - \left(\partial_{\xi_2 \xi_2}v_k+\partial_{\xi_2}v_k i\left.\omega_2^{1,2}\right\vert_{\Gamma_o}\right) e^{i \omega \xi_1} =- (\partial_{\xi_2 \xi_2}v_k\pm\partial_{\xi_2}v_k\omega_1) e^{i \omega_1 \xi_1}, \\
\partial_{\xi_2}\lambda_k'[\alpha] = - \left(\partial_{\xi_2 \xi_2}\lambda_k+\partial_{\xi_2}\lambda_k i\left.\omega_2^{1,2}\right\vert_{\Gamma_o}\right) e^{i \omega \xi_1} = - (\partial_{\xi_2 \xi_2}\lambda_k\pm\partial_{\xi_2}\lambda_k\omega_1) e^{i \omega \xi_1},
\end{align*}
where we used the expression for $\omega_2^{1,2}$ on the boundary, which was given in \eqref{eq:omega12OnGamma}. If we now plug this into the perturbed gradient and assume that $\omega_2^p$ and $\omega_2^a$ have the same sign, we get
\begin{align*}
df_D'[\alpha] = \mu \left[ \partial_{\xi_2}v_k \partial_{\xi_2 \xi_2}\lambda_k + \partial_{\xi_2}\lambda_k \partial_{\xi_2 \xi_2}v_k \pm2\omega_1(\partial_{\xi_2} \lambda_k \partial_{\xi_2} v_k )\right] \alpha.
\end{align*}
Hence, we have that the Hessian response to a Fourier mode with frequency $\omega$ is
\begin{align*}
H[\alpha]:= df_D'[\alpha] = (\beta_1 + \beta_2 \omega_1 ) \alpha.
\end{align*}
If we transform $\beta_1$ and $\beta_2$ back into physical coordinates, we get
\begin{align*}
\beta_1 =& \mu \sum_{k = 1}^2 \partial_{\xi_2}v_k \partial_{\xi_2 \xi_2}\lambda_k + \partial_{\xi_2}\lambda_k \partial_{\xi_2 \xi_2}v_k \nonumber \\
=& \mu (\bm{n}\cdot\nabla) \left(\sum_{k = 1}^2 (\bm{n}\cdot\nabla)v_k (\bm{n}\cdot\nabla)\lambda_k \right) 
\end{align*}
and
\begin{align*}
\beta_2 = \pm 2\mu \sum_{k = 1}^2 (\bm{n}\cdot\nabla) \lambda_k (\bm{n}\cdot\nabla) v_k .
\end{align*}
Note that if we use $\omega_2 = i\omega_1$, the perturbation of the state variables will go to zero for $\xi_2\rightarrow\infty$. This solution is plausible, as a perturbation of the surface should not change the flow solution far away from the obstacle. Therefore, we choose the sign to be negative. \qed
\end{proof}
Before constructing a preconditioner with the derived Hessian information, let us take a closer look at several interesting properties of the problem. With the help of the symbol, we can see that a Newton-like preconditioner will be important when trying to solve the optimization problem efficiently: If we follow \cite{ta1995trends} and interpret the symbol of the Hessian as an approximation of the eigenvalues, we see that the eigenvalues will grow linearly by a factor of $\beta_2$. Due to the fact that a fine discretization allows high as well as low frequencies in the design space, we obtain small and large eigenvalues, leading to an ill-conditioned Hessian. Consequently, steepest-descent methods will suffer from poor convergence rates, see \cite[Chapter~3.3]{wright1999numerical}. Furthermore, the Hessian symbol reveals the following properties:
\begin{enumerate}[I]
\item \label{itm:phase} $H[\alpha]$ is a wave with the same phase and frequency $\omega$ as $\alpha$.
\item \label{itm:linScaling} as the frequency of $\alpha$ increases, the amplitudes of $H[\alpha]$ increase linearly (linear scaling)
\item \label{itm:beta1beta2Scaling} the scaling consists of a constant part $\beta_1$ given by \eqref{eq:beta1} and a linear part $\beta_2$, which can be calculated according to \eqref{eq:beta2}
\item \label{itm:nonConstantScaling} the amplitudes of $H[\alpha]$ vary along $\xi_1$ as $\beta_1(\xi_1)$ and $\beta_2(\xi_1)$ are non-constant functions
\item the inverse of the Hessian will damp frequencies by a factor of
\begin{align*}
\sigma_{H^{-1}} = \frac{1}{\beta_1 + \beta_2 \omega},
\end{align*}
meaning that the inverse Hessian, which we wish to use as a preconditioner has smoothing behavior.
\end{enumerate}
Before turning to the construction of a preconditioner, we investigate the applicability of the analytic results for convective flows. \\

\section{The discrete Hessian symbol}
\tikzstyle{block} = [rectangle,draw,minimum width=8em,align=center,rounded corners, minimum height=2em,scale=1.0]
\tikzstyle{connect} = [draw,-latex']
\label{sec:section4}
In the following, we wish to numerically reproduce the analytically derived symbol to test the applicability of the analytic Hessian behavior in the case of convective terms. The calculations are carried out with the SU2 flow solver, which incorporates an optimization framework. Information as well as test cases of the SU2 solver can for example be found in \cite{palacios2013stanford}.\\
We look at a cylinder as described in section \ref{sec:section6} placed inside a flow with a Reynolds number of $1$ as well as $80$. For our configuration, these choices of the Reynolds number are reasonable, since a higher Reynolds number will result in an unsteady von K\'{a}rm\'{a}n vortex street, meaning that the derivation of our optimization framework no longer holds. The task is to change the shape of the cylinder such that the drag is minimized. Hence, the optimization patch $\Gamma_o$ is the surface of the cylinder. Due to the fact that we do not want to focus on the optimization, but on the Hessian approximation and especially its response to certain Fourier modes, we first think of possibilities to numerically determine the Hessian matrix. One way to do so is by finite differences. If the perturbed optimization patch is given by
\begin{align}\label{eq:pertSurf}
\Gamma_o^\epsilon(\xi_1) := \Gamma_o(\xi_1) + \epsilon \alpha(\xi_1)\bm{n}(\xi_1),
\end{align}
the shape Hessian in direction $\alpha$ is given by
\begin{align*}
H[\alpha] = \lim_{\epsilon\rightarrow 0} \frac{df_D(\Gamma_o^\epsilon)- df_D(\Gamma_o)}{\epsilon},
\end{align*}
where $df_D(\Gamma_o^\epsilon)$ is the gradient evaluated for the flow around the perturbed optimization patch \eqref{eq:pertSurf}. The dependency on the direction $\bm{V}_k$ as defined in \eqref{eq:Vk} has been omitted for better readability. The Hessian can now be approximated with finite differences, i.e., instead of taking the limit, we choose a small value for $\epsilon$, yielding
\begin{align} \label{eq:diffH}
H^{FD}[\alpha] := \frac{df_D(\Gamma_o^\epsilon)- df_D(\Gamma_o)}{\epsilon}.
\end{align}
We expect the numerical results to coincide with the analytic derivation, which is why we wish to recover the Hessian properties \ref{itm:phase} to \ref{itm:nonConstantScaling}.
\subsection{Flow case 1: Re = 1}
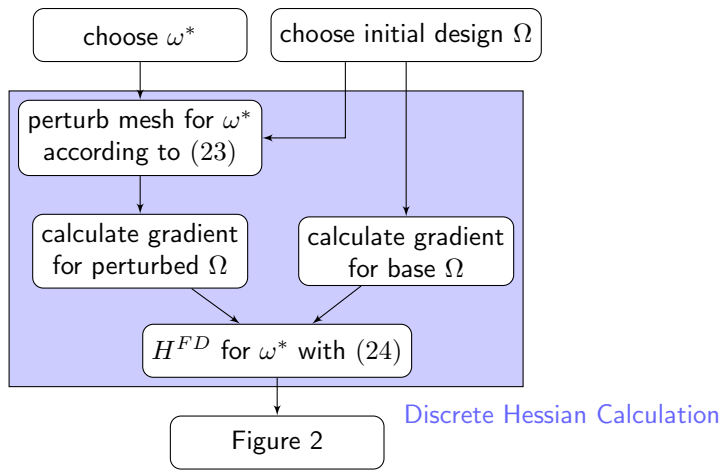
\begin{figure}[htbp!]
\begin{center}
\begin{tikzpicture}[node distance = 3.5cm,auto,font=\sffamily]
\node[block](id1){choose $\omega^*$};
\node[block, right of = id1](id2){choose initial design $\Omega$};
\node[block, below = 0.5cm of id1, fill=white](id3){perturb mesh for $\omega^*$\\ according to \eqref{eq:pertSurf}};
\node[block, below = 0.5cm of id3, fill=white](id4){calculate gradient \\ for perturbed $\Omega$};
\node[block, right of = id4, fill=white](id5){calculate gradient\\ for base $\Omega$};
\node[block, below left = 0.5cm and -1.5cm of id5,fill=white](id6){$H^{FD}$ for $\omega^*$ with \eqref{eq:diffH}};
\node[block, below = 0.5cm of id6](id7){Figure \ref{fig:inout}};
\begin{scope}[on background layer]
\node[fit =(id3)(id4)(id5)(id6), fill = blue!20, draw](box){};
\end{scope}
\node[below right = 0.1cm and -1.7cm of box]{\color{blue!60}{Discrete Hessian Calculation}};
\path[connect] (id1) -- (id3);
\path[connect] (id2.south) + (-0.8,0) |- (id3.east);
\path[connect] (id2) -- (id5);
\path[connect] (id3) -- (id4);
\path[connect] (id4) -- (id6);
\path[connect] (id5) -- (id6);
\path[connect] (id6) -- (id7);
\end{tikzpicture}
\end{center}
\caption{Scheme for computing the Hessian for specified frequency and design.}
\label{tz:DiscreteHessianResponse}
\end{figure}
\begin{figure}[htbp!]
  \centering
  \includegraphics[width=1.0\linewidth]{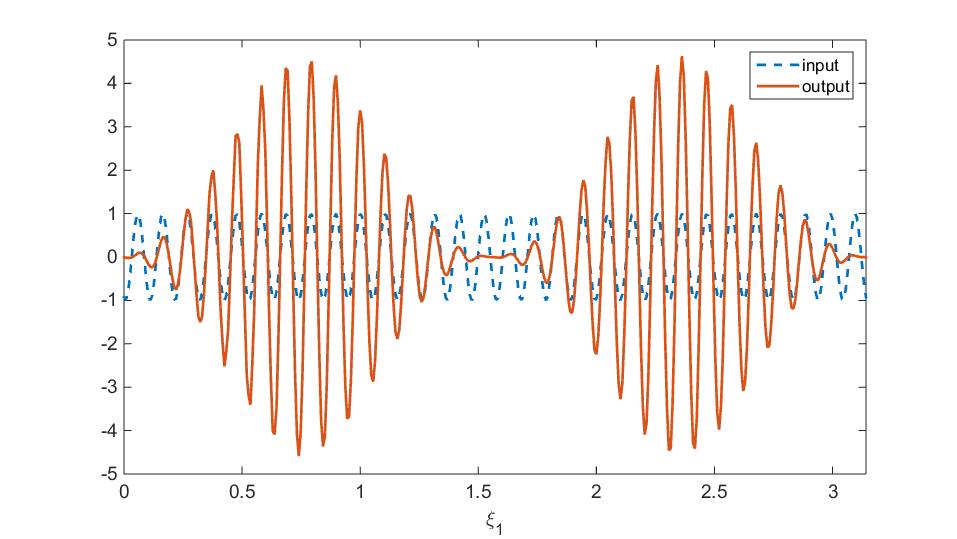}
  \caption{Scaled input $\alpha$ and resulting discrete Hessian $H^{FD}$ on the cylinder's surface.}
  \label{fig:inout}
\end{figure}
Let us start with a Reynolds number of $1$. The discrete Hessian response is calculated according to Figure \ref{tz:DiscreteHessianResponse} for $\omega^{*}=60$ and a two-dimensional cylinder as initial design. The resulting discrete shape Hessian can be seen in Figure \ref{fig:inout}. One can see that the Hessian structure coincides with the analytic results to the extent that the output will have the same phase and frequency as the input (property \ref{itm:phase}). Furthermore, we see that the Hessian will modify the amplitude of $\alpha$, which varies along the optimization patch. This behavior can also be deduced from the analytic derivation, as the non-constant derivatives of the primal and adjoint states affect the parameters $\beta_1$ and $\beta_2$ (property \ref{itm:nonConstantScaling}). 

\begin{figure}[htbp!]
\begin{center}
\begin{tikzpicture}[node distance = 3.5cm,auto,font=\sffamily]
\node[block](id0){choose $\xi_1^*$};
\node[block, below = 0.5cm of id0, fill=white](id1){choose $\omega_n$ for $n = 1,\cdots,N$ \\ s.t. $\sin(\omega_n\xi_1^*+s_n)\stackrel{!}{=}1$};
\node[block, right = 0.5cm of id0](id2){choose initial design $\Omega$};
\node[minimum height=2, below right = 0.5cm and -1.5cm of id1, fill = blue!20, draw](box){Discrete Hessian Calculation};
\node[block, left = 0.5cm of box, fill=white](out){Figure \ref{fig:Scalings}};
\node[block, below = 0.5cm of box, fill=white](out1){plot $H^{FD}$ at $\xi_1^*$ for all $\omega_n$};
\node[block, below = 0.5cm of out1, fill=white](out2){compute linear fit\\ for data points };
\node[block, left = 0.5cm of out2, fill=white](out3){Figure \ref{fig:AmplitudeScalingRe1}};
\node[block, below = 0.5cm of out2, fill=white](end){$\beta_1^{FD},\beta_2^{FD}$ at $\xi_1^*$};
\begin{scope}[on background layer]
\node[fit =(id1)(box)(out)(out1)(out2), fill = red!20, draw](newBox){};
\end{scope}
\node[below left = 0.1cm and -4.0cm of newBox]{\color{red!60}{Scaling Parameter Calculation}};
\path[connect] (id0) -- (id1);
\path[connect] (id1) -- (box);
\path[connect] (id2) -- (box);
\path[connect,dashed] (box.west) -- (out.east);
\path[connect] (box) -- (out1);
\path[connect] (out1) -- (out2);
\path[connect,dashed] (out1.west) -| (out3.north);
\path[connect,dashed] (out2) -- (out3);
\path[connect] (out2) -- (end);
\end{tikzpicture}
\end{center}
\caption{Scheme for calculating scaling parameters.}
\label{tz:SchemeScalingParameters}
\end{figure}
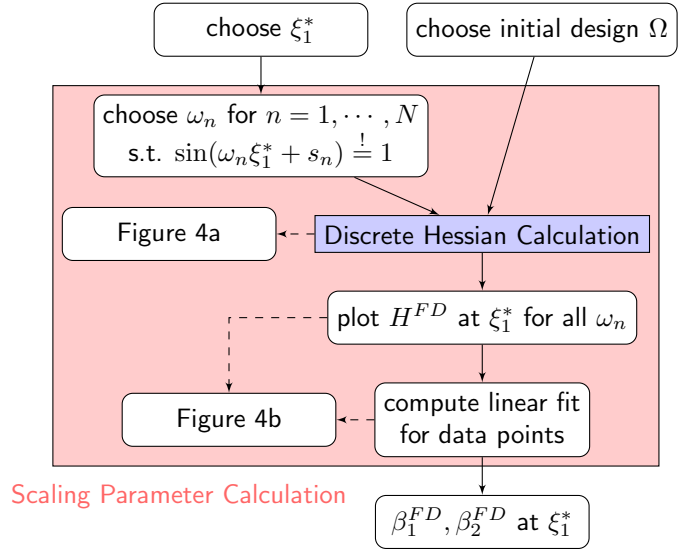
A detailed picture of how the amplitude depends on the input frequency at a fixed point $\xi_1^*$ can be obtained with the scheme depicted in Figure \ref{tz:SchemeScalingParameters}. We choose $N$ different frequencies such that the amplitudes of $\sin(\omega_n\xi_1+s_n)$ overlap at $\xi_1^{*}$. Note that we use a shift $s_n$ to allow choosing all frequencies. The resulting Hessian responses can be found in Figure \ref{fig:Scalings}.
\begin{figure}[htbp!]
        \centering
        \begin{subfigure}[b]{0.5\textwidth}
                \includegraphics[width=\textwidth]{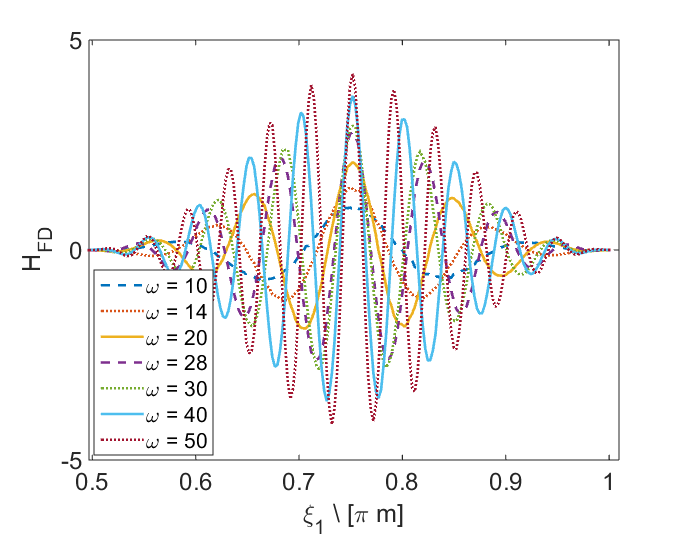} 
                \caption{Hessian responses on top cylinder.}
                \label{fig:Scalings}
        \end{subfigure}%
        ~ 
        \begin{subfigure}[b]{0.5\textwidth}
                \includegraphics[width=\textwidth]{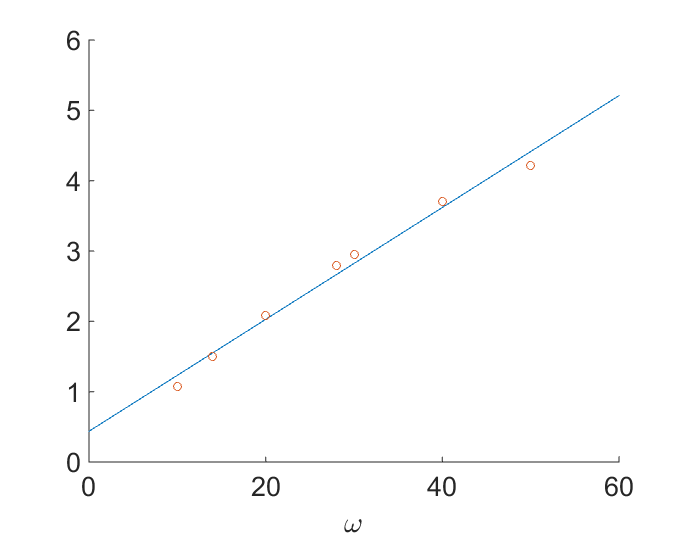}
                \caption{Amplitudes at $\xi_1 = \frac{3}{4}\pi$ with linear fit.}
                \label{fig:AmplitudeScalingRe1}
        \end{subfigure}
        ~ 
        \caption{Discrete Hessian responses to different input frequencies for $Re=1$.}
\end{figure}
We now evaluate the $N$ discrete Hessian responses at $\xi_1^*$. Plotting the different amplitudes over the corresponding input frequency $\omega$ and calculating a linear curve fit yields Figure \ref{fig:AmplitudeScalingRe1}. One can see that choosing a linear function will lead to a good approximation of the scaling behavior, indicating that the the numerical investigation matches the linear scaling of the analytically derived symbol (property \ref{itm:linScaling}). Note that the curve fit in Figure \ref{fig:AmplitudeScalingRe1} can be used to calculate the scaling parameters $\beta_1^{FD}$ (which is the fit at $\omega = 0$) as well as $\beta_2^{FD}$ (which is the slope of the fit) at $\xi_1=\frac{3}{4}\pi$. The superscript $FD$ denotes that the scaling parameters are obtained by the finite difference approximation and not by the analytic formulas \eqref{eq:beta1} and \eqref{eq:beta2}. 

\begin{figure}
\begin{center}
\begin{tikzpicture}[node distance = 3.5cm,auto,font=\sffamily]
\node[block](id0){choose $\xi_{1}^{(j)}$\\ for $j = 1,\cdots,M$};
\node[block, right = 0.5cm of id0](id1){choose initial design $\Omega$};
\node[minimum height=2, below right = 0.5cm and -1.5cm of id0, fill = red!20, draw](box){Scaling Parameter Calculation};
\node[block, below = 0.5cm of box](end){Figure \ref{fig:Re1Beta1Beta2}};
\path[connect] (id0) -- (box);
\path[connect] (id1) -- (box);
\path[connect] (box) -- (end);
\end{tikzpicture}
\end{center}
\caption{Scheme for calculating Figure \ref{fig:Re1Beta1Beta2}.}
\label{tz:SchemeBetaField}
\end{figure}
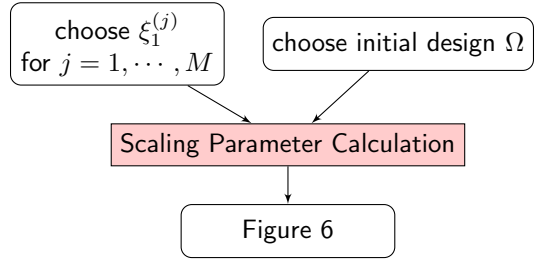

To calculate the scaling parameters $\beta_{1,2}^{FD}$ at several positions on the cylinder's surface, we repeat the previously described analysis for $M$ different choices of $\xi_1^*$, see \ref{tz:SchemeBetaField}. In Figure \ref{fig:Re1Beta1Beta2}, we compare the resulting scaling values with the continuous derivations \eqref{eq:beta1} and \eqref{eq:beta2}. Note that in order to minimize computational costs, we use only two frequencies, hence $N=2$.
\begin{figure}[htbp!]
  \centering
  \includegraphics[width=0.6\linewidth]{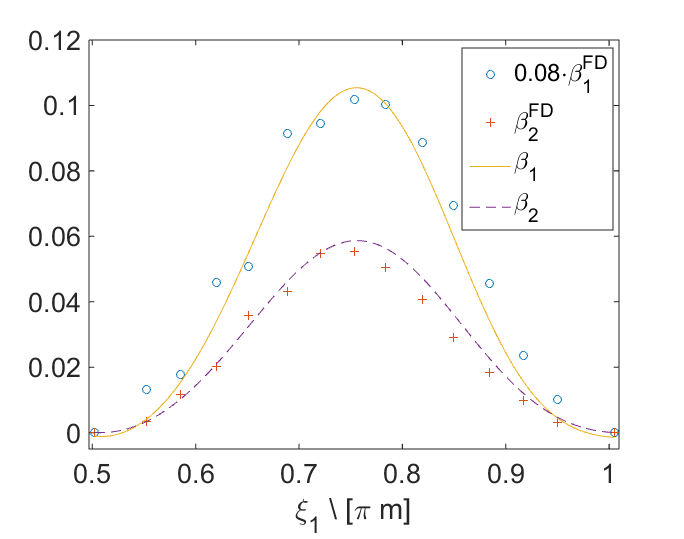}
  \caption{Comparison of the analytic result and the finite difference approximation of $\beta_{1,2}$ for $Re = 1$.}
  \label{fig:Re1Beta1Beta2}
\end{figure}
We can now see that the scaled values of $\beta_1^{FD}$ as well as the $\beta_2^{FD}$ values match the analytic results (property \ref{itm:beta1beta2Scaling}). Note that $\beta_1$ only coincides up to a factor of $0.08$, which is most likely caused by the poor approximation of second-derivatives of the flow solution.

As the properties of the finite difference approximation of the Hessian coincide with the analytic derivation, it is reasonable to use the analytic formulas of the $\beta$ values in order to calculate the preconditioner for small Reynolds number flows.

\subsection{Flow case 2: Re = 80}
We now turn to a flow with a Reynolds number of $80$. Again, we choose a surface perturbation $\alpha(\xi_1) = \cos(60(\xi_1-0.06))$ and investigate the resulting discrete Hessian, which can be seen in Figure \ref{fig:inoutRe80}. 
\begin{figure}[htbp!]
  \centering
  \includegraphics[width=1.0\linewidth]{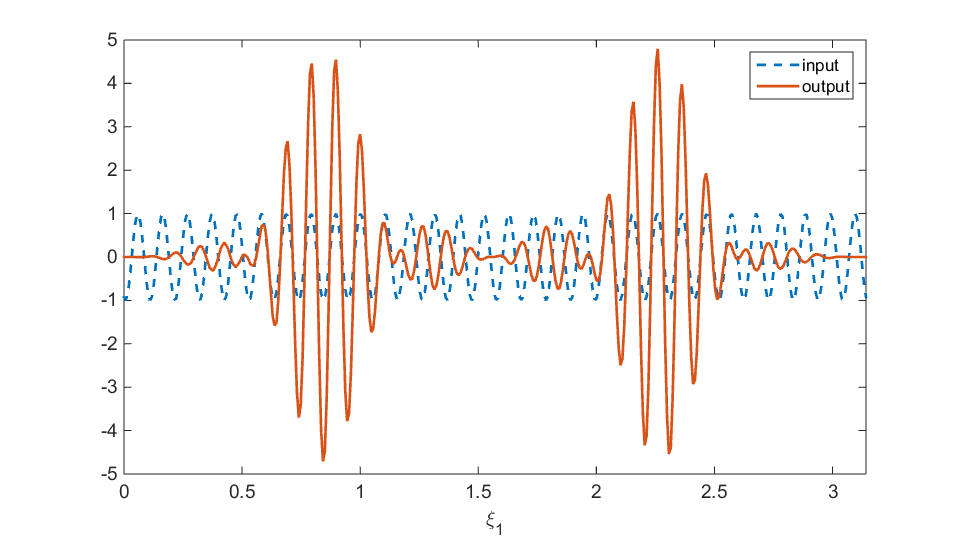}
  \caption{Scaled input $\alpha$ and resulting discrete Hessian $H^{FD}$.}
  \label{fig:inoutRe80}
\end{figure}
Just as in the first flow case, the input wave has the same phase as the outgoing Hessian signal (property \ref{itm:phase}). The amplitude of the output does again vary, meaning that we again have non-constant scaling parameters (property \ref{itm:nonConstantScaling}). The next step is to investigate how the output depends on the input frequency. We therefore study the output for several input frequencies and calculate a linear fit for the scaling of the amplitude, hoping that the linear analytic result will hold even though we no longer have negligible convective properties of the flow. The result of this fit at the spatial position $\xi_1 = \frac{3}{4}\pi$ can be found in Figure \ref{fig:AmplitudesRe80}.
\begin{figure}[htbp!]
        \centering
        \begin{subfigure}[b]{0.5\textwidth}
                \includegraphics[width=\textwidth]{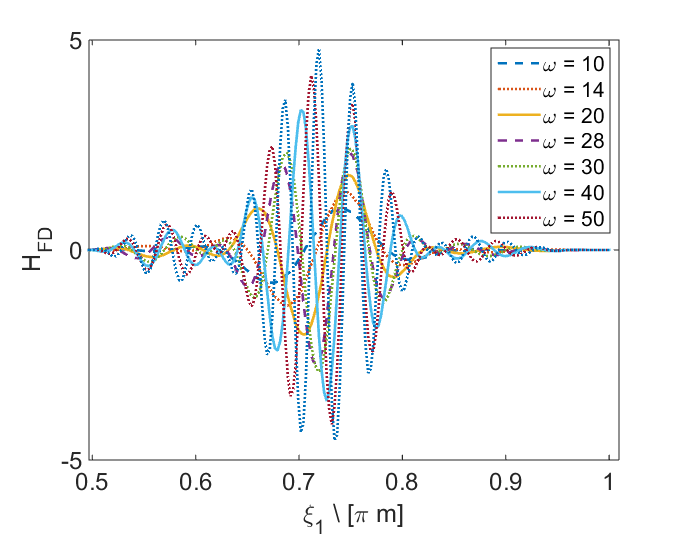} 
                \caption{Hessian responses on top cylinder.}
                \label{fig:ScalingsRe80}
        \end{subfigure}%
        ~ 
        \begin{subfigure}[b]{0.5\textwidth}
                \includegraphics[width=\textwidth]{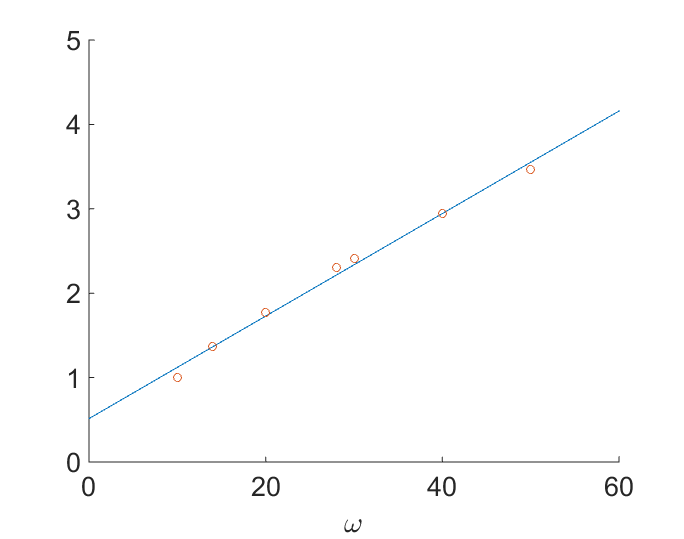}
                \caption{Amplitudes at $\xi_1 = \frac{3}{4}\pi$ with linear fit.}
                \label{fig:AmplitudesRe80}
        \end{subfigure}
        ~ 
        \caption{Discrete Hessian responses to different input frequencies for $Re=80$.}
\end{figure}
Fortunately, the results again point to a Hessian symbol with linear scaling (property \ref{itm:linScaling}). Repeating this computation for different values of $\xi_1$ yields Figure \ref{fig:Re80Beta1Beta2}.
\begin{figure}[htbp!]
  \centering
  \includegraphics[width=0.6\linewidth]{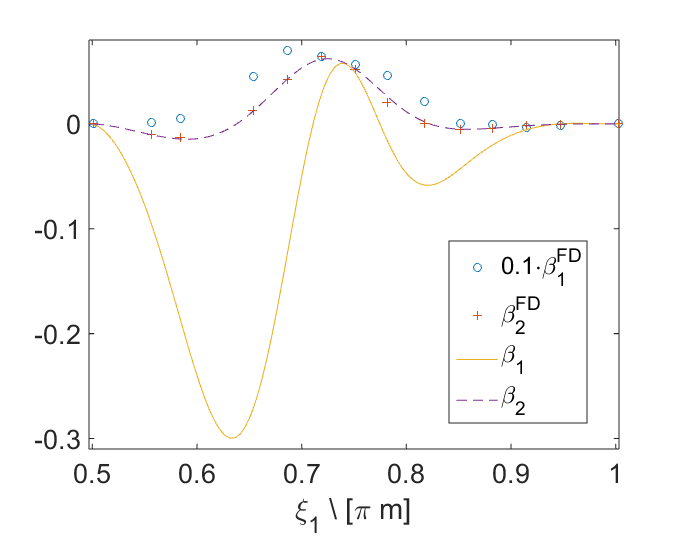}
  \caption{Comparison of the analytic result and the finite difference approximation of $\beta_{1,2}$ for $Re = 80$.}
  \label{fig:Re80Beta1Beta2}
\end{figure} 
It can be seen that the $\beta_2^{FD}$ values match the analytic result very well, whereas the $\beta_1^{FD}$ values do not coincide with the analytic predictions (property \ref{itm:beta1beta2Scaling} partially violated). Therefore, one can conclude that the convective flow behavior, which we did not include in the analytic derivations, will result in $\beta_1^{FD}$ values that do not correspond to $\beta_1$. However the analytic prediction of the scaling parameter $\beta_2$ can be used to mimic the Hessian behavior.

\section{Construction of the approximate Newton smoothing method}
\label{sec:section5}
Our aim is to use the scaling behavior, which we have investigated analytically and numerically to precondition and to smooth the search direction of our problem. Here, we need to distinguish between the low and higher Reynolds number cases, due to the fact that the numerical evaluation of $\beta_1$ did not coincide with the analytic prediction in the case of convective flow behavior. Let us for now assume that we know the values of $\beta_1$ and turn to several other problems arising when trying to determine a preconditioner. We start by using standard Hessian manipulation techniques as they can be found in \cite[Chapter~6.3]{wright1999numerical} to construct a modified Hessian $\bar{H}$, which is sufficiently positive definite. After that, we think of how to approximate this Hessian with a sparse and computationally cheap preconditioner $B$. Here, the main task will be to mimic pseudo-differential behavior.
\subsection{Hessian manipulation}
Let us start by pointing out that instead of using the Hessian, Newton's method uses the inverse Hessian, which has the inverse scaling behavior
\begin{equation}\label{eq:numericInvHessian}
H^{-1}[\alpha] = \frac{1}{\beta_1+\beta_2 \omega }\alpha.
\end{equation}
Our first step is to investigate the effect of this inversion, which can be found in Figure \ref{fig:inoutHInv} when using the analytic scaling parameters $\beta_{1,2}$ of a flow with a Reynolds number of $1$.
\begin{figure}[htbp!]
        \centering
        \begin{subfigure}[b]{0.5\textwidth}
                \includegraphics[width=\textwidth]{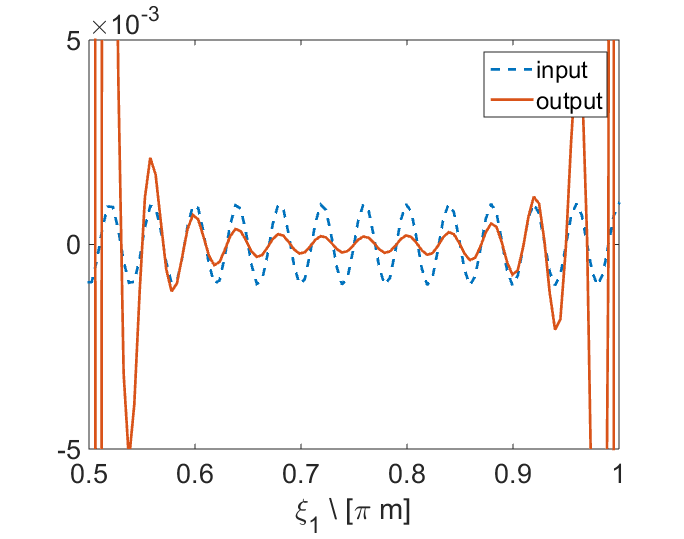}
                \caption{$H^{-1}[\alpha]$}
                \label{fig:inoutHInv}
        \end{subfigure}%
        ~ 
        \begin{subfigure}[b]{0.5\textwidth}
                \includegraphics[width=\textwidth]{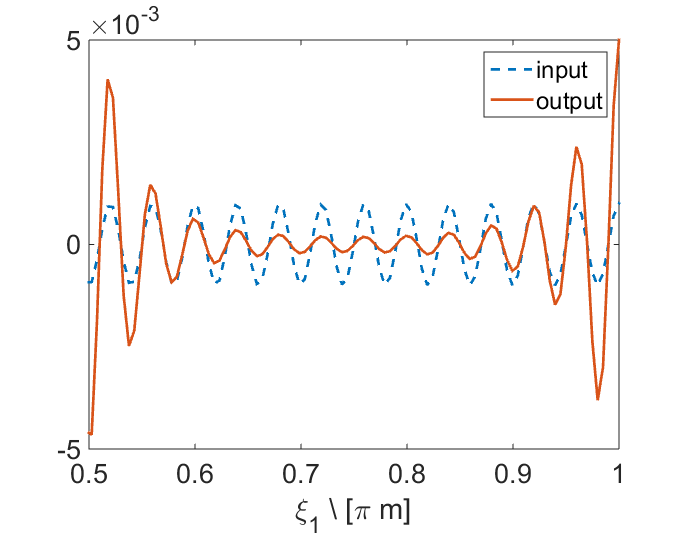}
                \caption{$\bar{H}^{-1}[\alpha], \eta = 0.2$}
                \label{fig:inoutHInvEps}
        \end{subfigure}
        ~ 
        \caption{Response of the original and modified preconditioner to the input $\alpha$, which is a wave with frequency $60$.}
\end{figure}
It is clear that the inverse Hessian will blow up for frequencies that fulfill
\begin{align*}
\beta_1+\beta_2\omega = 0.
\end{align*}
In our example, this behavior can be seen at the front and the rear of the cylinder. Furthermore, the scaling behavior of the Hessian can become negative, which can be interpreted as negative eigenvalues of the Hessian matrix. Hence, we need to take care of two problems frequently arising when trying to approximate a Hessian matrix, namely singularities as well as negative eigenvalues. As proposed in \cite{wright1999numerical}, we modify the Hessian symbol such that its inverse has the form
\begin{align*}
\bar{H}^{-1}[\alpha] := \frac{1}{\bar{\beta_1}+\left\vert\beta_2\right\vert \omega  }\alpha,
\end{align*}
where 
\begin{equation}\label{eq:choiceBeta1}
\bar{\beta_1}:=\eta+\beta_1-\min\left\{0,\min_{\xi_1}{\beta_1}\right\}.
\end{equation}
The regularization parameter $\eta$ is chosen to prevent singularities of the inverse symbol and to ensure that all eigenvalues of the Hessian will be sufficiently positive. Let us now use $\bar{H}^{-1}$ to precondition the search direction, which for now is a Fourier mode with frequency $60$. When taking a look at the output in Figure \ref{fig:inoutHInvEps}, we see that the Hessian behavior will only be affected in the critical regions of the optimization patch, namely the front and rear. We will further justify this modification of the Hessian by investigating its effects on the final preconditioner.
However, we first think of how one can calculate the Hessian when not using a Fourier mode as input. 
\subsection{Approximation of the pseudo-differential operator}
Our goal is to find a computationally cheap preconditioner $B^{-1}$ with the symbol $\sigma_{B^{-1}}$, which approximates the symbol of the inverted modified Hessian $\bar{H}^{-1}$. The two main properties of the symbol, namely its scaling and no phase shift belong to so-called pseudo-differential operators. Solving equations containing pseudo-differential operators is time consuming, which is why we use a different approach. To ensure a sparse and computationally cheap Hessian, we make use of differential operators. Using operators of even order will prevent a phase shift, however yields incorrect scaling. The chosen operator is
\begin{align*}
\bar{H}^{-1}[\alpha] \approx B^{-1}[\alpha] :=\left( \bar{\beta_1} + \epsilon \partial_{\xi_1 \xi_1} \right)^{-1} \alpha.
\end{align*}
This operator can easily be evaluated, however its symbol is
\begin{align}\label{eq:correctSymbol}
\sigma_{B^{-1}} = \frac{1}{\bar{\beta_1} - \epsilon \omega^2},
\end{align}
whereas the symbol which we wish to approximate is
\begin{align}\label{eq:approximatedSymbol}
\sigma_{\bar{H}^{-1}} = \frac{1}{\bar{\beta_1} + \beta_2 \omega}.
\end{align}
We will mimic the correct scaling by choosing $\epsilon$ such that the symbol of the preconditioner $\sigma_{B^{-1}}$ will be similar to the correct symbol $\sigma_{\bar{H}^{-1}}$. Before deriving a strategy to pick $\epsilon$, we apply the preconditioner $B$ to Newton's method: The preconditioned search direction $\bm{p}$ is given by
\begin{align*}
\bm{p}(\xi_1) = -\left( \bar{\beta_1}(\xi_1) + \epsilon(\xi_1) \partial_{\xi_1 \xi_1} \right)^{-1}\bm{df}(\xi_1),
\end{align*}
which is the continuous version of the Newton update \eqref{eq:NewtonDirection} when using the derived Hessian approximation. Discretizing this differential equation on the given mesh nodes of the optimization mesh yields
\begin{equation}\label{eq:discretePrecond}
\bar{\beta}_{1,j}p_{j} +  \frac{\epsilon_j}{\Delta\xi_1^2}\left( p_{j-1} -2p_{j}+p_{j+1} \right) = df_j
\end{equation}
with $j = 1,\cdots,N$. Since this discretized equation is linear in $\bm{p}$, we can rewrite it as
\begin{align}\label{eq:localPreconditioning}
\bm{p} = -\bm{B}^{-1}\bm{df},
\end{align}
where the matrix $\bm{B}$ is given by
\begin{align*}
\begin{pmatrix}
\bar{\beta}_{1,1}-\frac{2\epsilon_{1}}{\Delta\xi_1^2} & \frac{\epsilon_{1}}{\Delta\xi_1^2} & 0 &  &  & \dots &  & 0 & \frac{\epsilon_{1}}{\Delta\xi_1^2}      \\
\frac{\epsilon_{2}}{\Delta\xi_1^2}	& \bar{\beta}_{1,2}-\frac{2\epsilon_2}{\Delta\xi_1^2} & \frac{\epsilon_2}{\Delta\xi_1^2}	& 0 &  & & & \dots & 0 	  \\
0 & \frac{\epsilon_3}{\Delta\xi_1^2}	& \bar{\beta}_{1,3}-\frac{2\epsilon_3}{\Delta\xi_1^2} & \frac{\epsilon_3}{\Delta\xi_1^2}	& 0 & & & \dots  & 0 	  \\
\vdots	& &	& \ddots & & & & & \vdots \\
0	& &	&  & & & & &  \\
\frac{\epsilon_{N}}{\Delta\xi_1^2} 	& 0 & \dots & & & & 	 & \frac{\epsilon_{N}}{\Delta\xi_1^2}	& \bar{\beta}_{1,N}-\frac{2\epsilon_{N}}{\Delta\xi_1^2}
\end{pmatrix},
\end{align*}
when assuming periodic boundary conditions and the gradient $\bm{df}$ is the collection of gradient values on every surface node as defined in \eqref{eq:GradientVector}. 

We now return to choosing the smoothing parameter $\epsilon$ such that the correct and approximated symbols \eqref{eq:approximatedSymbol} and \eqref{eq:correctSymbol} match for relevant frequencies. Note that $\epsilon$ has been discretized in \eqref{eq:discretePrecond}. Hence, it remains to pick $\epsilon_j$ for $j = 1,\cdots,N$. For this task, we need to determine frequencies in the gradient vector $\bm{df}$. We use the discrete Fourier transform
\begin{align*}
df_k = \frac{1}{N}\sum_{l = 0}^{N-1}\hat{df}_l\text{exp}\left(\frac{i2\pi k l}{N}\right)
\end{align*}
to determine frequencies with amplitude $\hat{df}_l$ in $\bm{df}$. These amplitudes can be calculated by 
\begin{align*}
\hat{df}_k = \sum_{l = 0}^{N-1} df_l \text{exp}\left(\frac{i2\pi k l}{N}\right).
\end{align*}
Note that these frequencies are global. Frequencies can be localized by multiplying a discrete window function $g$, yielding
\begin{align}\label{eq:windowedFourier}
\tilde{df}_{m,l} = \sum_{k = 0}^{N-1}df_k g_{k-m}\text{exp}\left(\frac{-i2\pi k l}{N}\right).
\end{align}
This representation is common in signal compression and is called a discrete windowed Fourier transform. For further details can be found in \cite[Chapter~4.2.3]{mallat1999wavelet}.

The discrete values of $\bm{\epsilon}$ are now determined by minimizing the distance between the response of the windowed Fourier transform to the correct and approximated symbol
\begin{equation}\label{eq:minScale}
\epsilon_j = \argmin_{\varepsilon} \sum_{k=0}^{N-1} \tilde{df}_{j,k}^2\left(\frac{1}{\bar{\beta}_{1,j} - \varepsilon \omega_k^2 } -\frac{1}{\bar{\beta}_{1,j} + \vert\beta_{2,j}\vert \omega_k }\right)^2.
\end{equation}
Since we have localized the frequencies, we are able to pick the smoothing parameter in a given spatial cell $j$ such that the approximated matches the correct scaling for frequencies that are dominant in cell $j$. The optimal value of this parameter in cell $j$ is then denoted by $\epsilon_j$. The constructed preconditioner  is depicted in Figure \ref{tz:ApproxNewtonSmoothing}.

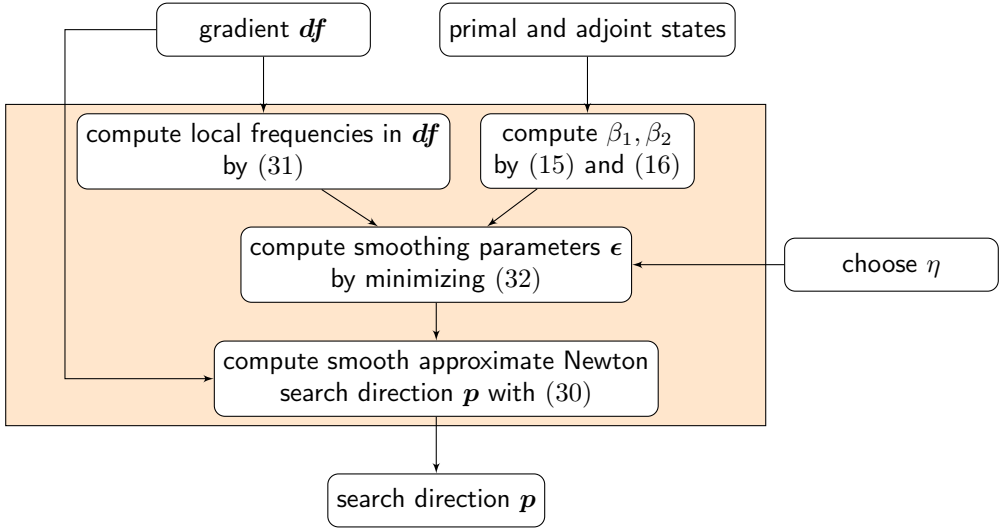
\begin{figure}
\begin{center}
\begin{tikzpicture}[node distance = 3.5cm,auto,font=\sffamily]
\node[block](grad){gradient $\bm{df}$};
\node[block, right =0.9cm of grad](states){primal and adjoint states};
\node[block, below = 0.75cm of grad, fill=white](FF){compute local frequencies in $\bm{df}$ \\ by \eqref{eq:windowedFourier}};
\node[block, below = 0.75cm of states, fill=white](beta){compute $\beta_1,\beta_2$ \\ by \eqref{eq:beta1} and  \eqref{eq:beta2}};
\node[block, below left = 0.5cm and -2.0cm of beta, fill=white](epsilon){compute smoothing parameters $\bm{\epsilon}$ \\ by minimizing \eqref{eq:minScale}};
\node[block, right = 2.0cm of epsilon, fill=white](eta){choose $\eta$};
\node[block, below = 0.5cm of epsilon, fill=white](p){compute smooth approximate Newton \\ search direction $\bm{p}$ with \eqref{eq:localPreconditioning}};
\begin{scope}[on background layer]
\node[minimum width=10cm, fit =(FF)(beta)(epsilon)(p), fill = orange!20, draw](newBox){};
\end{scope}
\node[block, below = 0.75cm of p](output){search direction $\bm{p}$};
\path[connect] (grad) -- (FF);
\path[connect] (states) -- (beta);
\path[connect] (beta) -- (epsilon);
\path[connect] (FF) -- (epsilon);
\path[connect] (eta) -- (epsilon);
\path[connect] (epsilon) -- (p);
\path[connect] (grad.west) -| node [near start] {} ([xshift=-1.2cm] grad.west)
						  |- ([xshift=-.0cm] p.west);
\path[connect] (p) -- (output);
\end{tikzpicture}
\end{center}
\caption{Approximate Newton smoothing method}
\label{tz:ApproxNewtonSmoothing}
\end{figure}

The presented preconditioner is similar to common Sobolev smoothing: The Sobolev-smoothed search direction $p_S$ is obtained by solving
\begin{align}\label{eq:SobolevSmoothing}
p_S(\xi_1) = -\left( 1 + \tilde{\epsilon} \partial_{\xi_1 \xi_1} \right)^{-1}df(\xi_1).
\end{align}
In contrast to the presented method, the smoothing parameter $\tilde{\epsilon}$ is usually obtained by a parameter study. In our method, we pick the spatially dependent smoothing parameter $\epsilon$ such that we mimic Hessian behavior. Hence, the introduced smoothness is chosen locally such that the optimization process is accelerated. Therefore, we call the new method \textit{local smoothing}, whereas common Sobolev Smoothing is called \textit{global smoothing} in the following.

We now investigate the effects of the Hessian manipulations introduced by the modification of the scaling parameters $\beta_1$ and $\beta_2$. Remember that our goal was to construct a sufficiently positive definite preconditioner, which means that the smallest eigenvalue $\mu^*$ fulfills
\begin{align*}
\mu^* \geq \delta > 0.
\end{align*}
We can estimate the eigenvalues of the preconditioning matrix $\bm{B}$ with the help of the Gershgorin circle theorem, which states that
\begin{align*}
\vert \mu - B_{kk} \vert < \sum_{j\neq k} \vert B_{kj}\vert.
\end{align*}
Therefore, we have that
\begin{align*}
\mu \in \bigcup_{k} \left( \bar{\beta}_{1,k}-2\frac{\epsilon_k}{\Delta \xi_1^2} - 2\frac{\vert\epsilon_k\vert}{\Delta \xi_1^2}, \bar{\beta}_{1,k}-2\frac{\epsilon_k}{\Delta \xi_1^2} + 2\frac{\vert\epsilon_k\vert}{\Delta \xi_1^2} \right) .
\end{align*}
Note that since we are using $\vert\beta_1\vert$ as scaling parameter, we know that $\epsilon_k < 0 $, which means that we have
\begin{align*}
\mu \in \bigcup_{k} \left(\bar{\beta}_{1,k}, \bar{\beta}_{1,k}-4\frac{\epsilon_k}{\Delta \xi_1^2} \right).
\end{align*}
Remember that in \eqref{eq:choiceBeta1} the modified scaling parameter $\bar{\beta_1}$ was chosen such that $\bar{\beta_1} > \eta$, meaning that the regularization parameter $\eta$ can be understood as the minimal eigenvalue of the preconditioner. Hence the regularization parameter should be chosen sufficiently large, meaning that $\eta \geq \delta$. As a result we can easily control the lower bound of the minimal eigenvalue. This further motivates the Hessian modifications we used. We now use the preconditioner, which we constructed to optimize the design of a cylinder inside a flow. 

\section{Results}
\label{sec:section6}
In the following, the results of the optimization when using common preconditioner \eqref{eq:localPreconditioning} will be presented and compared to the common choice of Sobolev smoothing \eqref{eq:SobolevSmoothing} with a constant smoothing parameter $\tilde\epsilon$, which we call global preconditioning. A good smoothing parameter for the global method, has been determined by investigating the grid resolution as done in \cite{schmidt2009impulse}. As discussed, the local methods picks the smoothing parameter automatically by minimizing \eqref{eq:minScale}. The task is to minimize the drag of a two dimensional cylinder, which is placed inside a fluid. To prevent the methods from simply decreasing the volume of the cylinder in order to achieve a minimization of the drag, we employ a volume constraint, which ensures a constant obstacle volume. The radius of the cylinder is one meter. We use a farfield density of $998.2\frac{kg}{m^3}$ and a farfield velocity of $10^{-5}\frac{m}{s}$ for the Reynolds number of $1$ and a velocity of $6.4\cdot 10^{-5}\frac{m}{s}$ for the Reynolds number of $80$. The chosen viscosity is $0.798\cdot 10^{-3}\frac{Ns}{m^2}$.
\subsection{Flow case 1: Re = 1}
We first look at the flow with $Re=1$, where we are confident to use the analytic form of $\beta_2$ and especially $\beta_1$. The regularization parameter $\eta$ in \eqref{eq:choiceBeta1} is set to $0.2$. As commonly done in Newton's method, we use a step length of $1.0$ for the local method. When comparing the first design update of the local and the global method when using a step length of $1$, one observes that the global method is penalized as the design update is much smaller. This is why we scale the step size of the global method, such that the magnitude of the design change will be of the same size for both methods in the first design step, see Figure \ref{fig:searchIt1Re1}. Let us now compare the optimization histories of local and global preconditioning. In Figure \ref{fig:fRe1}, we can see that using the analytic derivation of the scaling parameters $\beta$ as well as the information on the local frequencies inside the gradient will lead to a speedup, compared to the common global preconditioner. Whereas global preconditioning needs $15$ iterations to decrease the drag by roughly six percent, the local method will reach this reduction after nine iterations. A comparison of the flow field before and after the optimization can be found in Figure \ref{fig:v1OptNewMethodRe1}.
\begin{figure}[htbp!]
        \centering
        \begin{subfigure}[b]{0.55\textwidth}
                \includegraphics[width=\textwidth]{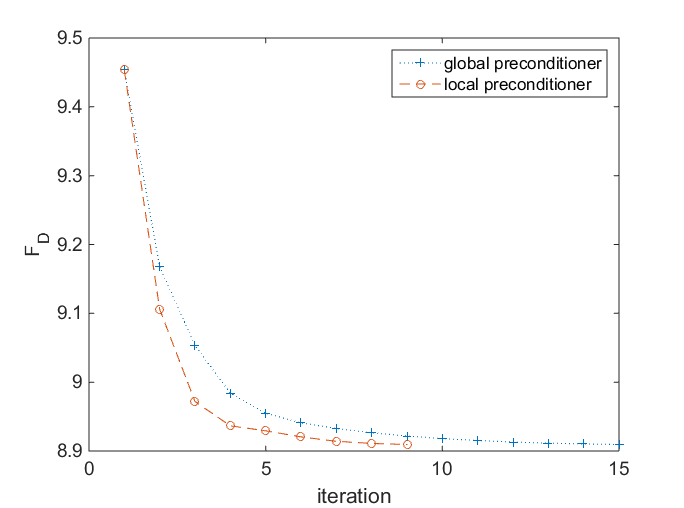} 
                \caption{Optimization history}
                \label{fig:fRe1}
        \end{subfigure}%
        \begin{subfigure}[b]{0.55\textwidth}
                \includegraphics[width=\textwidth]{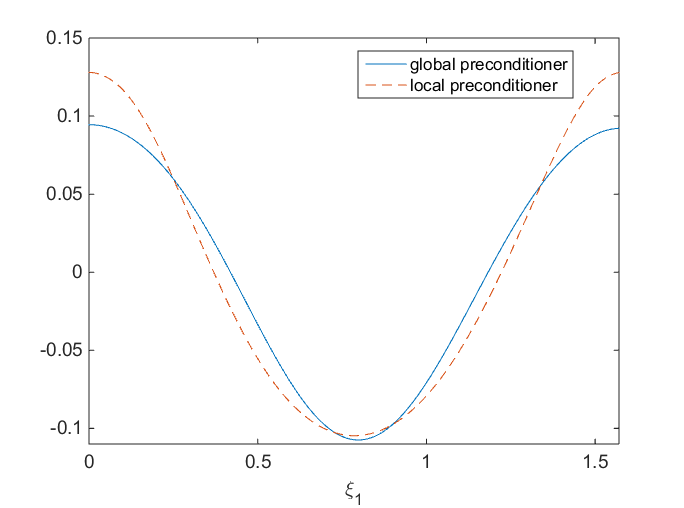}
                \caption{Scaled search direction iteration $1$.}
                \label{fig:searchIt1Re1}
        \end{subfigure}
        \caption{Comparison of standard and local preconditioning for $Re=1$.}
\end{figure}
\begin{figure}[htbp!]
  \centering
  \includegraphics[width=0.7\linewidth]{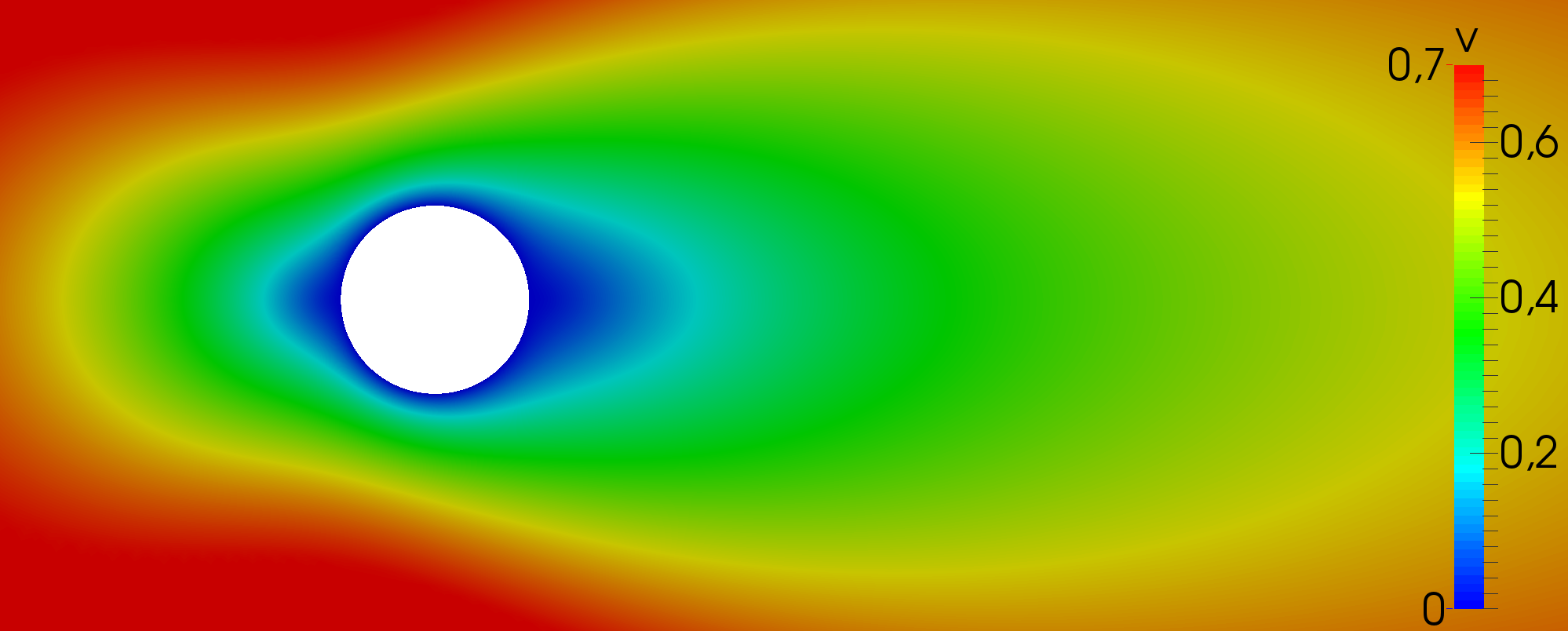}
  \includegraphics[width=0.7\linewidth]{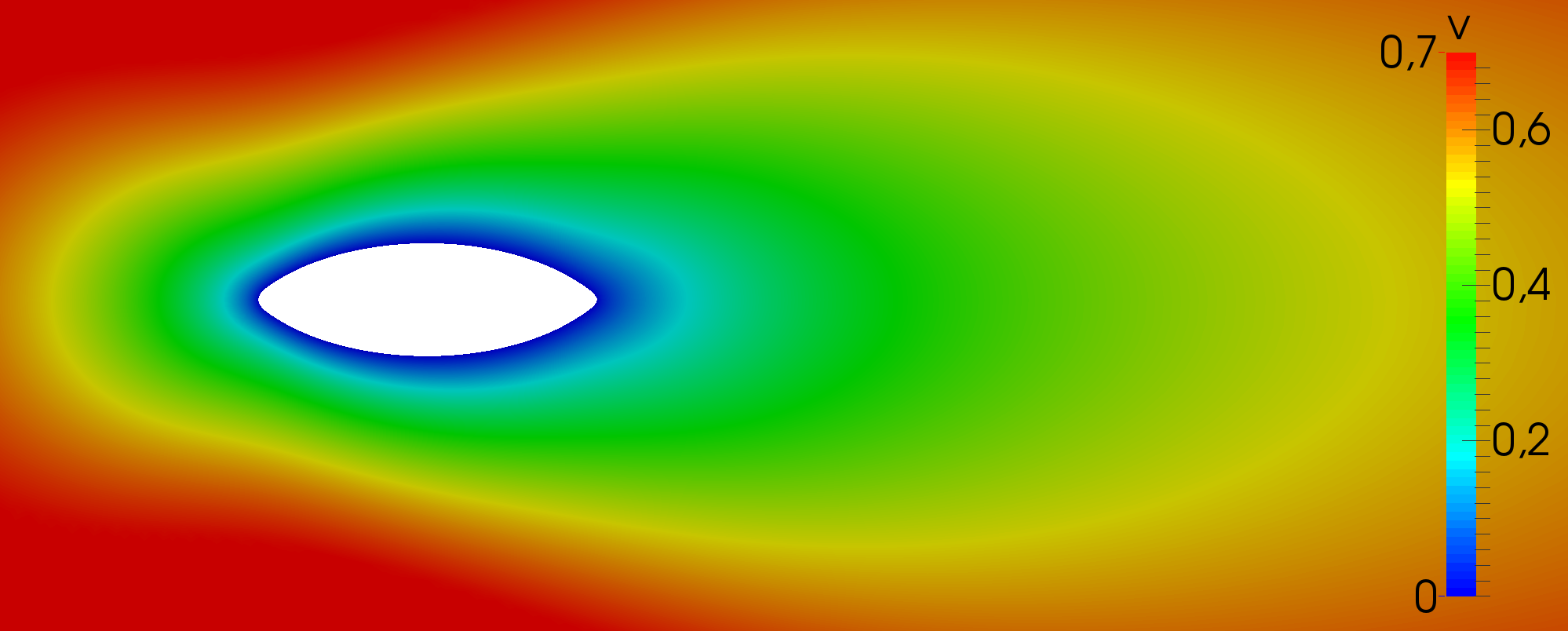}
    \caption{Zoomed view of the flow solution. Top: Original design and velocity magnitude. Bottom: Locally optimized design and velocity magnitude.}
  \label{fig:v1OptNewMethodRe1}
\end{figure}
\subsection{Flow case 2: Re = 80}
We now turn to the more complicated convective case, where we choose a Reynolds number of $80$. Again, we modify the surface of a cylinder in order to reduce the drag. Remembering the numerical investigation of the Hessian matrix for such a flow, it is obvious, that we cannot use the analytic values of $\beta_1$ as scaling parameter, whereas the values of $\beta_2$ fit the analytic prediction. A reasonable choice for $\beta_1$ is a scaled and smoothed version of $\beta_2$, which can be seen when looking at the numeric results of Figure \ref{fig:Re80Beta1Beta2}. Therefore, we choose $\beta_1 = 10\beta_2$, where the scaling of $10$ is motivated by the value $\beta_1^{FD}$, which we calculated when choosing multiple frequencies, see Figure \ref{fig:ScalingsRe80}. Furthermore, we use Sobolev smoothing with a very small choice of the smoothing parameter $\epsilon\approx 10^{-4}$ to ensure a smooth scaling parameter $\beta_1$. The regularization parameter $\eta$ is chosen as in the first flow case, meaning that a value of $0.2$ is taken. Due to the fact that we wish to use a constant step length for the optimization process, we use a step length of $0.5$ for the local preconditioner and scale the search direction proposed by the global preconditioner such that both search directions are of the same size in the first optimization step. A comparison of the scaled search directions can be found in Figure \ref{fig:searchIt1Re80}. 
\begin{figure}[htbp!]
        \centering
        \begin{subfigure}[b]{0.55\textwidth}
                \includegraphics[width=\textwidth]{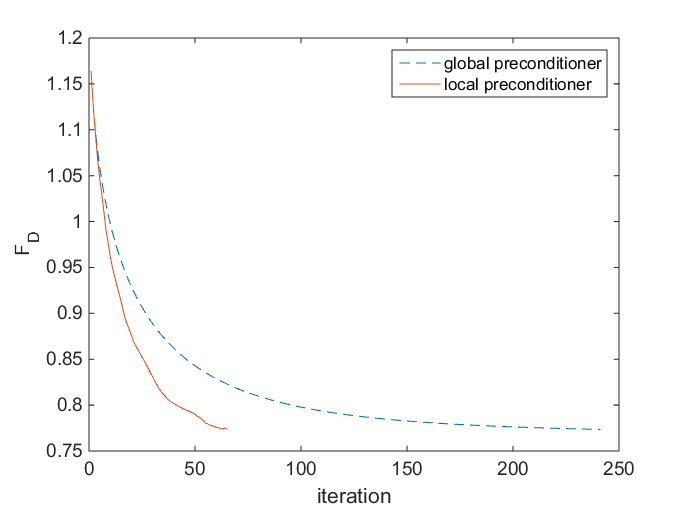} 
                \caption{Optimization history}
                \label{fig:fRe80}
        \end{subfigure}%
        \begin{subfigure}[b]{0.55\textwidth}
                \includegraphics[width=\textwidth]{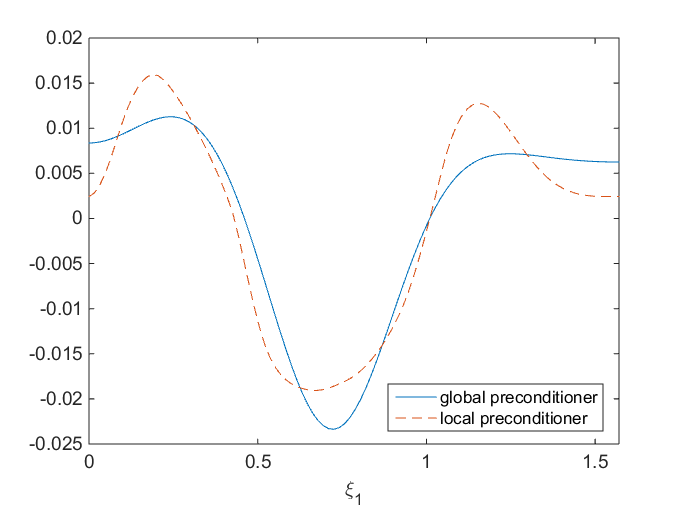}
                \caption{Scaled search direction iteration $1$.}
                \label{fig:searchIt1Re80}
        \end{subfigure}
        \caption{Comparison of standard and local preconditioning for $Re=80$.}
\end{figure}
\begin{figure}[htbp!]
  \centering
  \includegraphics[width=0.7\linewidth]{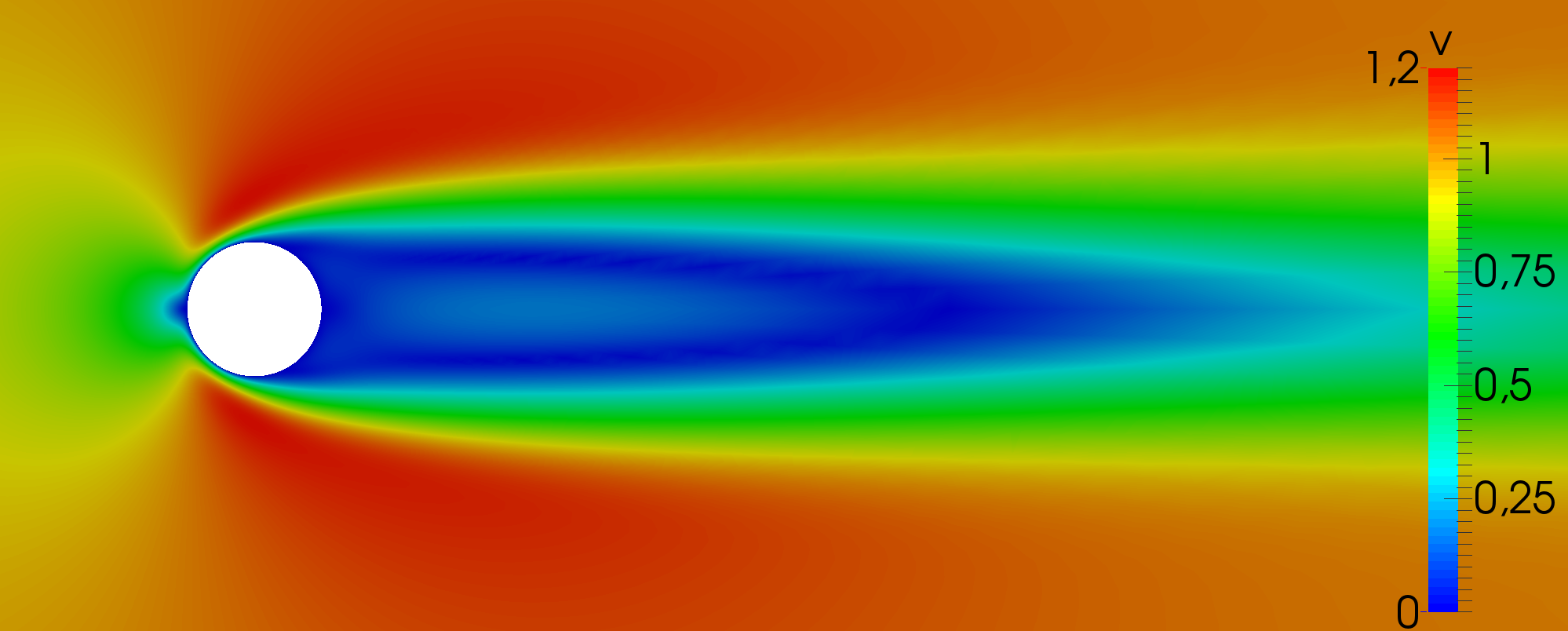}
  \includegraphics[width=0.7\linewidth]{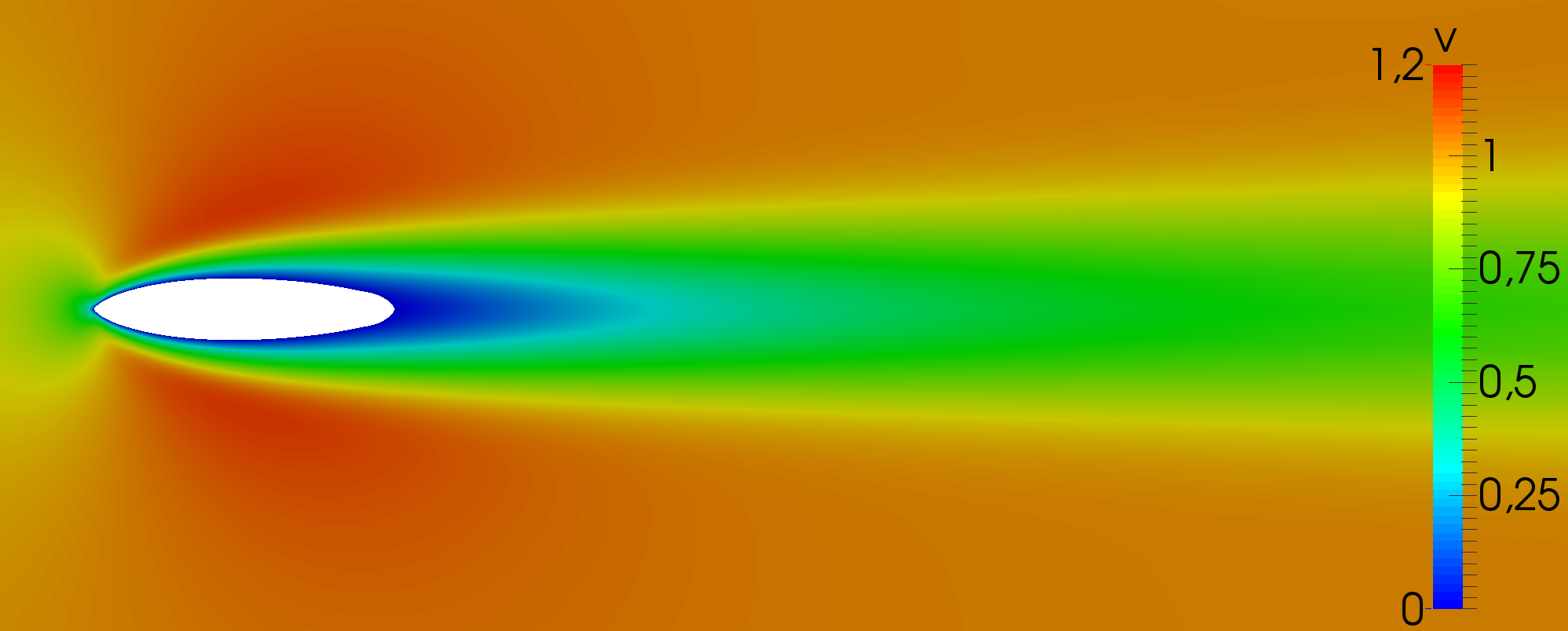}
  \caption{Zoomed view of the flow solution. Top: Original design and velocity magnitude. Bottom: Locally optimized design and velocity magnitude.}
  \label{fig:v1OptNewMethodRe80}
\end{figure}
\begin{figure}[htbp!]
  \centering
  \includegraphics[width=0.8\linewidth]{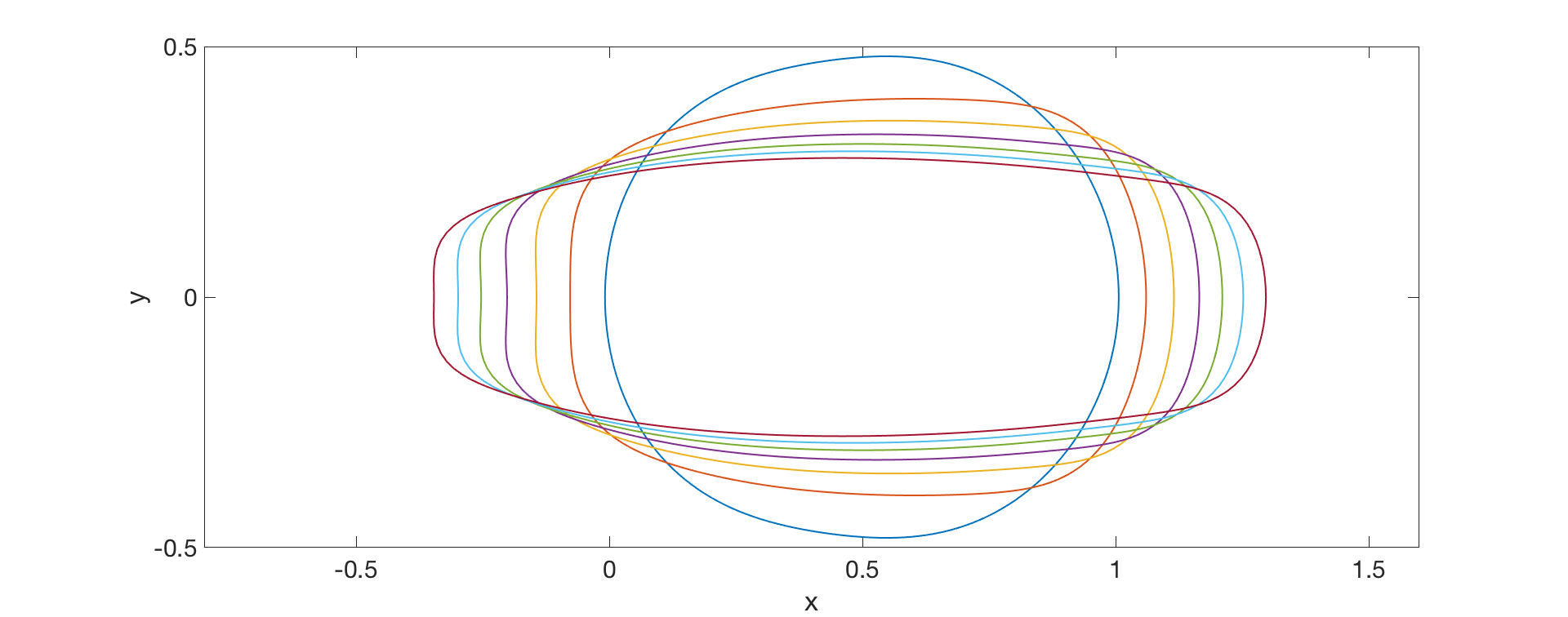}
  \includegraphics[width=0.8\linewidth]{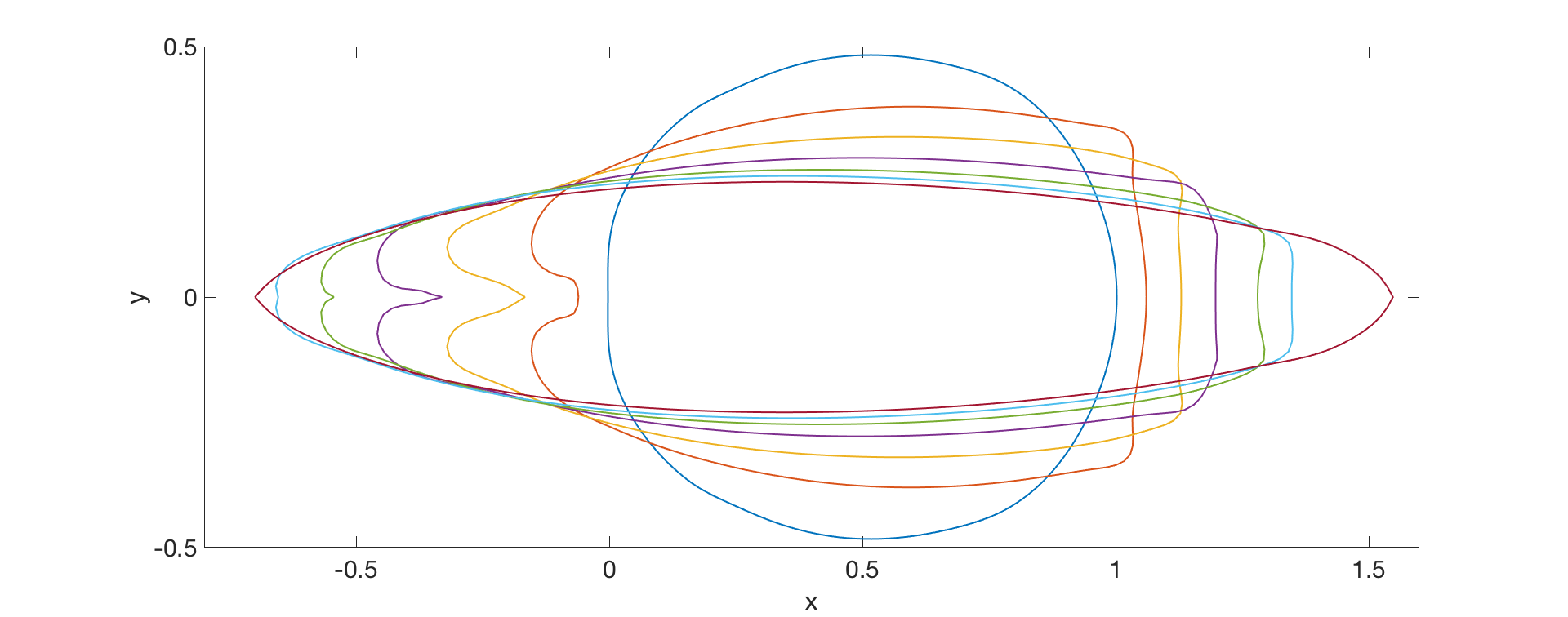}
  \caption{Obstacle designs at iteration $1,10,20,\cdots ,60$. Top: Global preconditioning. Bottom: Local preconditioning.}
  \label{fig:optimizationDesign}
\end{figure}
During the optimization process, we see that the local preconditioner will focus on creating an optimal front and rear, whereas the global preconditioner will heavily modify the top and bottom of the obstacle. Let us now compare the optimization histories of the local and the global preconditioner, which can be seen in Figure \ref{fig:fRe80}. We see that both methods are able to heavily decrease the drag by more than $33$ percent. While the local preconditioner will reach this reduction after $62$ iterations, the global preconditioner needs $240$ iterations to reach the same drag value. It is important to note that the local preconditioner will not further decrease the drag value after iteration $62$, since the step size is too big. A smaller step size will further decrease the drag value, however we wish to perform our optimization with a constant step length, which is why subsequent drag values do not appear in the optimization history. The global method is stopped after iteration $240$ as the norm of the gradient will fall to almost zero. Let us now take a closer look at the optimized design, which can be found in Figure \ref{fig:v1OptNewMethodRe80}. We can see that the optimization process will create a sharp front and rear, as well as a smooth top and bottom. Comparing the design histories of the global and local method in Figure \ref{fig:optimizationDesign}, we can see that the local preconditioner will focus on creating an optimal front and rear, whereas the global preconditioner will heavily modify the top and bottom of the obstacle. Note that the local method results in an inverted front, in the first design steps. However, the ability to choose non-smooth deformations is advantageous in this problem as a sharp edge is allowed to form. A disadvantage of non-smooth deformations is that it can lead to complex meshes, which is why robust mesh deformation tools need to be employed. 

\section{Summary and Outlook}
\label{sec:section7}
In this paper, we derived a local smoothing preconditioner, which automatically picks smoothing parameters such that the symbol of the inverse Hessian is approximated. This preconditioner has been derived by determining the analytic symbol when choosing the Stokes equations as flow constraints. The resulting coefficients of the symbol $\beta_1$ and $\beta_2$, which we called scaling parameters, have been compared to the numerical Hessian response. The presented technique to determine the scaling parameters numerically showed good agreements with the analytic results for flows with a Reynolds number of one. As convective forces become dominant, the parameter $\beta_1$ looses validity, however $\beta_2$ coincides with the analytic calculation. Standard Hessian manipulations of approximate Newton method have been used to obtain a sufficiently positive definite preconditioner. A computationally cheap preconditioner, which mimics the symbol of the Hessian, has been constructed by using differential operators. The derived method can be interpreted as Sobolev smoothing, which automatically picks a local smoothing parameter such that the symbol of the Hessian is approximated. Comparing the new method with Sobolev smoothing, we see that we obtain a faster convergence to the optimal design. By making use of a local smoothing parameter, which depends on the position of the optimization patch $\Gamma_o$, the method is able to turn off smoothing in physically meaningful areas such as the front and rear of the cylinder. 

A question that one could focus on in future work is how to determine the scaling parameter $\beta_1$ in the case of a convective flow. Setting $\beta_1$ to a smoothed version of $\beta_2$ led to an acceleration of the optimization, however this choice was based on problem dependent numerical investigations, which might not hold for further applications. Furthermore, one needs to check the validity of the Hessian symbol at non-smooth parts of the optimization patch, as the symbol has been derived for smooth geometries. A construction of further preconditioners making use of the derived Hessian symbol is possible. Here, one should look at the construction of a preconditioner with pseudo-differential properties to further improve the search direction.


\bibliography{Paper}

\end{document}